\documentclass[11pt,reqno]{amsart}

\usepackage[utf8]{inputenc}
\usepackage{setspace}
\usepackage{geometry}
\usepackage{enumerate}
\usepackage{enumitem, xcolor, amssymb,latexsym,amsmath,bbm}
\usepackage{mathtools}
\usepackage[mathscr]{euscript}
\usepackage{amsmath}
\usepackage{amssymb}
\usepackage{enumitem}

\usepackage{tikz}
\usepackage{wrapfig}
\usepackage{enumerate}
\usepackage{graphicx}
\usepackage{subfigure}
\usepackage{esint}
\usetikzlibrary{arrows.meta}
\usepackage{mathrsfs}
\usetikzlibrary{decorations.markings}
\usepackage[colorlinks=true,citecolor=blue,
linkcolor=blue,urlcolor=blue]{hyperref}

\newtheorem{theorem}{Theorem}
\newtheorem{theorem*}{Theorem}
\newtheorem{lemma}{Lemma}
\newtheorem{proposition}[theorem]{Proposition}
\newtheorem{definition}{Definition}
\newtheorem{corollary}{Corollary}
\newtheorem{remark}{Remark}

\newcommand{\RR}{\mathbb{R}}
\newcommand{\CC}{\mathbb{C}}
\newcommand{\NN}{\mathbb{N}}
\newcommand{\ZZ}{\mathbb{Z}}

\newcommand{\HH}{\mathbb{H}}

\numberwithin{equation}{section}
\numberwithin{theorem}{section}
\numberwithin{lemma}{section}
\numberwithin{example}{section}
\allowdisplaybreaks

\makeatletter
\@namedef{subjclassname@2020}{\textup{2020} Mathematics Subject Classification}
\makeatother

\begin{document}
    
\title[Fractional Caffarelli-Kohn-Nirenberg type inequalities on $\HH^n$]{Fractional Caffarelli-Kohn-Nirenberg type inequalities on the Heisenberg group}

    \author{Rama Rawat}
    \email{rrawat@iitk.ac.in}

    \author{Haripada Roy}
    \email{haripada@iitk.ac.in}

    \author{Prosenjit Roy}
    \email{prosenjit@iitk.ac.in}
    
    \address{Indian Institute of Technology Kanpur, India}

    \date{}

    \smallskip
\begin{abstract}
The aim of this work is to establish some cases of the  Caffarelli-Kohn-Nirenberg inequalities on the Heisenberg group for the fractional Sobolev spaces. Here we work with the fractional Sobolev spaces as given by Adimurthi and Mallick in \cite{Adi-Mallick}. Our inequalities also give an improvement on the range of indices for the Hardy type inequality established in \cite{Adi-Mallick}.

    \end{abstract}
    
    %\pagenumbering{roman}
    \maketitle
    
        %\dedicatory{}
        %\tableofcontents
        %\pagenumbering{arabic}
        %\setcounter{page}{1}
        %\pagestyle{myheadings}

\section{Introduction}
 In \cite{CKN}, L. Caffarelli, R. Kohn and L. Nirenberg established a family of interpolation inequalities on the Euclidean space, now known as, Caffarelli-Kohn-Nirenberg (CKN) inequalities. The Hardy and Sobolev inequalities are included as particular cases of these CKN inequalities. Later analogous versions of Hardy-Sobolev inequalities were obtained in \cite{Hitchhiker,Dyda,Ground-state,Mazya-S} in the setting of fractional Sobolev spaces on $\RR^n$, in \cite{Adi-Mallick} on the Heisenberg group and in \cite{Suragan-Hardy} for homogeneous spaces.\smallskip

In this paper we establish some cases of the fractional Caffarelli-Kohn-Nirenberg (CKN) inequalities on the Heisenberg group. We will be working with the fractional Sobolev spaces as defined by Adimurthi and Mallick in \cite{Adi-Mallick} which are more natural when the index $p$ is not necessarily $2.$ This differs from the fractional Sobolev spaces considered in \cite{Suragan-Hardy,Suragan-Ruz} and the results therein for more general homogeneous spaces including the Heisenberg group.
As in the Euclidean case, our version of CKN fractional inequalities contains in particular fractional Hardy's inequality and Sobolev inequality as proved in \cite{Adi-Mallick}. Further, we work with decomposition of the Heisenberg group into sets which allow us to estimate required integrals, circumventing the use of extension domains. We are not aware if half spaces or the domains which appear in our proofs are extension domains for the fractional Sobolev spaces we are considering in this paper. Our results also improve on range of indices for which the fractional Hardy's inequality holds.\smallskip

We begin by setting notations and recalling results which we need in the sequel. Let $n\geq1$, $p>1$, $s\in(0,1)$ and $\eta,\,\eta_1,\,\eta_2\in \RR$ with $\eta_1+\eta_2=\eta$. For any measurable subset $\Omega\subset \RR^n$, we define the fractional Sobolev space $W^{s,p,\eta}(\Omega)$ as
$$W^{s,p,\eta}(\Omega):=\left\{u\in L^p(\Omega): [u]_{W^{s,p,\eta}(\Omega)} :=\left(\int_{\Omega}\int_{\Omega}\frac{|x|^{\eta_1p}|y|^{\eta_2p}|u(x)-u(y)|^p}{|x-y|^{n+sp}}dxdy\right)^{\frac{1}{p}}<\infty\right\},$$
endowed with the norm 
$$\lVert u\rVert_{W^{s,p,\eta}(\Omega)}=\lVert u\rVert_{L^p(\Omega)}+[u]_{W^{s,p,\eta}(\Omega)}.$$
When $\eta_1=\eta_2=\eta=0$, we denote the quantities $\lVert u\rVert_{W^{s,p,0}(\Omega)}$ and $[u]_{W^{s,p,0}(\Omega)}$ by $\lVert u\rVert_{W^{s,p}(\Omega)}$ and $[u]_{W^{s,p}(\Omega)}$ respectively.\smallskip

Let $p>1$, $0<s<1$, $q\geq1$, $\tau>0$, $0\leq a\leq 1$ and $\beta,$ $\gamma,$ $\eta,$ $ \eta_1,$ $\eta_2$ $\in\RR$ be such that $\eta_1+\eta_2=\eta$ and
\begin{equation}\label{condition 1}
\frac{1}{\tau}+\frac{\gamma}{n}=a\left(\frac{1}{p}+\frac{\eta-s}{n}\right)+(1-a)\left(\frac{1}{q}+\frac{\beta}{n}\right).
\end{equation}
When $a>0$, we assume the following additional conditions
\begin{equation}\label{condition 2}
\gamma=a\sigma+(1-a)\beta,\ \ 0\leq \eta-\sigma
\end{equation}
and 
\begin{equation}\label{condition 3}
\eta-\sigma\leq s,\ \mathrm{when}\ \ \frac{1}{\tau}+\frac{\gamma}{n}=\frac{1}{p}+\frac{\eta-s}{n}.
\end{equation}
The following fractional Caffarelli-Kohn-Nirenberg inequalities were proved by Hoai-Minh Nguyen and Marco Squassina in \cite{frac-CKN}.
\begin{theorem*}\emph{\cite{frac-CKN}}\label{frac CKN}
Let $n,p,s,q,\tau,a, \beta,\gamma,\eta_1,\eta_2$ and $\eta$ be as above satisfying the conditions \eqref{condition 1}, \eqref{condition 2} and \eqref{condition 3}.\\
\emph{(i)} If $\frac{1}{\tau}+\frac{\gamma}{n}>0$, then we have
\begin{equation}\label{eq frac CKN a}
\lVert|x|^{\gamma}u\rVert_{L^{\tau}(\RR^n)}\leq C [u]^a_{W^{s,p,\eta}(\RR^n)} \lVert|x|^{\beta}u\rVert^{1-a}_{L^{q}(\RR^n)},\ \forall u\in C_c^1(\RR^n).
\end{equation}
\emph{(ii)} If $\frac{1}{\tau}+\frac{\gamma}{n}<0$, then we have
\begin{equation}\label{eq frac CKN b}
\lVert|x|^{\gamma}u\rVert_{L^{\tau}(\RR^n)}\leq C [u]^a_{W^{s,p,\eta}(\RR^n)} \lVert|x|^{\beta}u\rVert^{1-a}_{L^{q}(\RR^n)},\ \forall u\in C_c^1(\RR^n\setminus \{0\}).
\end{equation}
\end{theorem*}
If we consider $a=1$ and $\eta_1=\eta_2=\eta=0$ in \eqref{condition 1}, $\gamma=0$ implies $\tau=p^*=\frac{np}{n-p}$, and inequality \eqref{eq frac CKN a} gives the well known fractional Sobolev inequality proved earlier in \cite{Hitchhiker,Mazya-S}. On the other hand, the corresponding inequality for $\gamma=-s$ and $\tau=p$ is the fractional Hardy inequality obtained in \cite{Ground-state,Mazya-S}.\smallskip

Inequalities \eqref{eq frac CKN a} and \eqref{eq frac CKN b} are not true in general for the case $\frac{1}{\tau}+\frac{\gamma}{n}=0$. Instead the following weighted version holds in this case as was proved in \cite{frac-CKN}.
\begin{theorem*}\emph{\cite{frac-CKN}}
Let $n,p,s,q,\tau,a, \beta,\gamma,\eta_1,\eta_2$ and $\eta$ be as above satisfying the condition \eqref{condition 1} and $0\leq \eta-\sigma\leq s$.\\
\emph{(i)} If $\frac{1}{\tau}+\frac{\gamma}{n}=0$ and $u\in C_c^1(\RR^n)$ with supp$(u)\subset B(0,R)$, then we have
\begin{equation}\label{log case a}
\left(\int_{\RR^n}\frac{|x|^{\tau\gamma}}{\ln^{\tau}\left(4R/|x|\right)}|u|^\tau dx\right)^{\frac{1}{\tau}}\leq C [u]^a_{W^{s,p,\eta}(\RR^n)} \lVert|x|^{\beta}u\rVert^{1-a}_{L^{q}(\RR^n)}.
\end{equation}
\emph{(ii)} If $\frac{1}{\tau}+\frac{\gamma}{n}=0$ and $u\in C_c^1(\RR^n)$ with supp$(u)\cap B(0,r)=\phi$, then we have
\begin{equation}\label{log case b}
\left(\int_{\RR^n}\frac{|x|^{\tau\gamma}}{\ln^{\tau}\left(4|x|/r\right)}|u|^\tau dx\right)^{\frac{1}{\tau}}\leq C [u]^a_{W^{s,p,\eta}(\RR^n)} \lVert|x|^{\beta}u\rVert^{1-a}_{L^{q}(\RR^n)}.
\end{equation}
\end{theorem*}
Before going further, we recall the preliminaries on the Heisenberg group. The Heisenberg group $\HH^n$ is defined as
$$\HH^n:=\{\xi=(x,y,t)=(z,t):(x,y)=z\in \RR^n\times \RR^n\ \mathrm{and}\ t\in \RR\}$$
with the group law
$$\xi\circ\xi'=(x+x',y+y',t+t'+2\langle y,x'\rangle-2\langle x,y'\rangle),$$
where $\xi=(x,y,t)$, $\xi'=(x',y',t')\in \HH^n$ and $\langle\cdot,\cdot\rangle$ denotes the usual Euclidean inner product in $\RR^n$. It is easy to see that $0\in\HH^n$ is the identity element and $-\xi$ is the inverse of $\xi\in\HH^n$. A basis for the left invariant vector fields is given by
$$X_i=\frac{\partial}{\partial x_i}+2y_i\frac{\partial}{\partial t},\ \ 1\leq i\leq n,$$
$$Y_i=\frac{\partial}{\partial y_i}-2x_i\frac{\partial}{\partial t},\ \ 1\leq i\leq n,$$
$$T=\frac{\partial}{\partial t}.$$
For any $C^1$ function $u$, $\nabla_{\HH^n}(u)$ is called the sub-gradient, defined by
$$\nabla_{\HH^n}(u)=\left(X_1(u),\cdots,X_n(u),Y_1(u),\cdots,Y_n(u)\right)$$
and
$$|\nabla_{\HH^n}(u)|^2=\sum_{i=1}^{n}\left(|X_i(u)|^2+|Y_i(u)|^2\right).$$
For any $\xi=(x,y,t)=(z,t)\in\HH^n$, $d(\xi)=\left((|x|^2+|y|^2)^2+t^2\right)^{\frac{1}{4}}=\left(|z|^4+t^2\right)^{\frac{1}{4}}$ is the Koranyi-Folland non isotropic gauge and $Q=2n+2$ is the homogeneous dimension of the Heisenberg group $\HH^n$. The left-invariant Haar measure on $\HH^n$ is the Lebesgue measure on $\RR^{2n+1}.$ We define the ball centered at $\xi$ with radius $r$ in $\HH^n$ as
\begin{equation}\label{set B_r}
B_r(\xi)=\{\xi'=(z',t'):d({\xi}^{-1}\circ \xi')<r\},
\end{equation}
and the annular domain as
\begin{equation}\label{set A_r,R}
A_{r,R}(\xi)=\{\xi'=(z',t'):r<d({\xi}^{-1}\circ \xi')<R\}.
\end{equation}
We also define the following set
\begin{equation}\label{set D_r}
D_r(\xi)=\{\xi'=(z',t'):|z'-z|<r\ \mathrm{and}\ \ |t'-t-2\langle y,x'\rangle+2\langle x,y'\rangle|<r^2\}.
\end{equation}

Various versions of Hardy's inequality on $\HH^n$ were proved by several authors in last two decades. The following inequality is one of the general version of the same, proved by L. D'Ambrozio in \cite{D'Ambrozio}:
\begin{equation}\label{eq D'Ambro}
\int_{\Omega}\frac{|z|^\beta}{d(\xi)^\alpha}|u(\xi)|^pd\xi\leq C_{n,p,\alpha,\beta}^p\int_{\Omega}|\nabla_{\HH^n}(u)|^p |z|^{\beta-p}d(\xi)^{2p-\alpha},\ \forall u\in C_c^{\infty}(\Omega),
\end{equation}
where $\Omega$ is an open subset of $\HH^n$, $n\geq1$, $p>1$, $\alpha,\beta\in\RR$ satisfy $Q>\alpha-\beta$ and $2n>p-\beta$ and $C_{n,p,\alpha\beta}=\frac{p}{Q+\beta-\alpha}$ is the optimal constant.\smallskip

Inequality \eqref{eq D'Ambro} with $\alpha=4$ and $\beta=2$ was proved in \cite{Garofalo} for the domain $\HH^n\setminus\{0\}$. Whereas \eqref{eq D'Ambro} was proved for the domain $\HH^n$ with $\alpha=2p$ and $\beta=p$ in \cite{Yong-Wang}. Adimurthi and Sekar proved \eqref{eq D'Ambro} with $\alpha=2p$ and $\beta=2$, for $u\in FS_0^{1,p}(\HH^n)$ in \cite{Adi-Sekar}, where the space $FS_0^{1,p}(\HH^n)$ is the completion of $C_c^{\infty}(\HH^n)$ under the norm
$$|u|^p_{1,p}=\int_{\HH^n}\frac{|\nabla_{\HH^n}(u)|^p}{|z|^{p-2}}dzdt.$$\smallskip

In \cite{Han}, Yazhou Han established Caffarelli-Kohn-Nirenberg type inequalities for $\HH^n$. Further, D. Suragan and his coauthors had extended some cases of CKN type inequalities on stratified Lie groups (see \cite{Suragan1,Suragan2}) and  homogeneous groups (see \cite{Suragan3}) with optimal constants.
\smallskip

For $1\leq p<\infty$, $s\in(0,1)$ and $\alpha\in\RR,$ consider the fractional order Sobolev space on $\HH^n$ as given in \cite{Adi-Mallick}:
$$W_0^{s,p,\alpha}(\HH^n)=\mathrm{Cl}\left\{f\in C_c^{\infty}(\HH^n): \int_{\mathbb{H}^n} \int_{\mathbb{H}^n}\frac{|f(\xi)-f(\xi')|^pd\xi' d\xi}{d({\xi}^{-1}\circ \xi')^{Q+ps} |z-z'|^{(p-2)\alpha}}<\infty\right\},$$
where $\xi=(z,t)$, $\xi'=(z',t').$ 
The closure is considered under the norm $\lVert f\rVert_{s,p,\alpha}^p=\lVert f\rVert_{L^p(\HH^n)}^p+[f]_{s,p,\alpha}^p$, where
$$[f]_{s,p,\alpha}=\left(\int_{\mathbb{H}^n} \int_{\mathbb{H}^n}\frac{|f(\xi)-f(\xi')|^pd\xi' d\xi}{d({\xi}^{-1}\circ \xi')^{Q+ps} |z-z'|^{(p-2)\alpha}}\right)^\frac{1}{p}.$$
Let $\Omega$ be any open set in $\HH^n$. We define the space
$$W^{s,p,\alpha}(\Omega)=\left\{f\in L^p(\Omega):[f]_{s,p,\alpha,\Omega}:=\left(\int_{\Omega} \int_{\Omega}\frac{|f(\xi)-f(\xi')|^pd\xi' d\xi}{d({\xi}^{-1}\circ \xi')^{Q+ps} |z-z'|^{(p-2)\alpha}}\right)^\frac{1}{p}<\infty\right\},$$
with the norm $\lVert f\rVert_{s,p,\alpha,\Omega}^p=\lVert f\rVert_{L^p(\Omega)}^p+[f]_{s,p,\alpha,\Omega}^p$.\smallskip

Throughout this work, we assume $(p-2)\alpha\geq 0$. We also assume $(p-2)\alpha<Q-2$, which follows from the non triviality of the space $W_0^{s,p,\alpha}(\HH^n)$ (see \cite{Adi-Mallick}, Section 3). The following Sobolev embedding theorem for $W_0^{s,p,\alpha}(\HH^n)$ is proved in \cite{Adi-Mallick}.

\begin{theorem*}\emph{\cite{Adi-Mallick}}\label{frac-sob in H^n}
Let $1\leq p<\infty$, $s\in(0,1)$ and $\alpha\in\RR$ be such that $ps+(p-2)\alpha<Q$. Then there exists a positive constant $C$, depends only on $n$, $p$, $s$ and $\alpha$, such that
\begin{equation}\label{eq frac-sob in H^n}
\lVert f\rVert_{L^{p^*}(\HH^n)}\leq C[f]_{s,p,\alpha},\  \forall f\in W_0^{s,p,\alpha}(\HH^n),
\end{equation}
where $p^*=\frac{Qp}{Q-ps-(p-2)\alpha}$ is the fractional critical Sobolev exponent.
\end{theorem*}
A fractional Hardy type inequality for the functions of $W_0^{s,p,\alpha}(\HH^n)$ has been proved in \cite{Adi-Mallick} which is as follows:
\begin{theorem*}\emph{\cite{Adi-Mallick}}\label{frac-har in H^n}
Let $1\leq p<\infty$, $s\in(0,1)$ and $\alpha\in\RR$ satisfy the following conditions:\\
\emph{(a)} $ps>2$;\\
\emph{(b)} $ps+(p-2)\alpha<Q$.\\
Then there exists a positive constant $C$, depends only on $n$, $p$, $s$ and $\alpha$, such that 
\begin{equation}\label{eq frac-har in H^n}
\int_{\HH^n}\frac{|f(\xi)|^pd\xi}{d(\xi)^{ps}|z|^{(p-2)\alpha}}\leq C[f]_{s,p,\alpha}^p,\  \forall f\in W_0^{s,p,\alpha}(\HH^n).
\end{equation}

\end{theorem*}
Theorem \ref{frac-har in H^n} holds without the condition (a) only for the case $p=2$ or $\alpha=0$ i.e., when $(p-2)\alpha=0$ (see \cite{Adi-Mallick}, Remark 1.3.). In \cite{Suragan-Hardy} A. Kassymov and D. Suragan established a fractional Hardy type inequality on the homogeneous Lie groups, which in particular on $\HH^n$ gives \eqref{eq frac-har in H^n} with $\alpha=0$. If $(p-2)\alpha>0$, condition (a) forces $p$ to be large enough when $s$ is very small. On the other hand, non triviality of the space restricts the range of $p$. To be precise, Theorem \ref{frac-har in H^n} remains unproven for $0<s\leq 2/\left(\frac{Q-2}{\alpha}+2\right)$.\smallskip

In this  paper we establish analogous version of the inequalities \eqref{eq frac CKN a}, \eqref{eq frac CKN b}, \eqref{log case a} and \eqref{log case b} (considering $\eta_1=\eta_2=\eta=\beta=0$) for $W_0^{s,p,\alpha}(\HH^n)$, which includes Theorem \ref{frac-sob in H^n} and Theorem \ref{frac-har in H^n} as particular cases. For the case $(p-2)\alpha=0$, an analogous version of the inequalities \eqref{eq frac CKN a}-\eqref{log case b} in the Heisenberg group were established in \cite{Suragan-Ruz}. However, $(p-2)\alpha\neq0$ case is totally different from the case $(p-2)\alpha=0$. In our case the corresponding dimensional balance relation is
\begin{equation}\label{tau gamma rel}
\frac{1}{\tau}+\frac{\gamma\left(1+(p-2)\frac{\alpha}{ps}\right)}{Q}=a\left(\frac{1}{p}-\frac{s\left(1+(p-2)\frac{\alpha}{ps}\right)}{Q}\right)+\frac{1-a}{q},
\end{equation}
where $p>1$, $s\in(0,1)$, $\alpha\in\RR$, $\tau>1$, $0<a\leq 1$, $\gamma\in \RR$ and $q\geq 1$. We also assume
\begin{equation}\label{range of gamma}
0\leq -\gamma\leq as.
\end{equation}
We prove the following results.
\begin{theorem}\label{frac CKN in H^n}
Let $p>1$, $s\in(0,1)$, $\alpha\in\RR$, $\tau>1$, $0<a\leq 1$, $\gamma\in\RR$ and $q>(1-a)\tau$ satisfy \eqref{tau gamma rel} and \eqref{range of gamma}.\\
\emph{(i)} If $\frac{1}{\tau}+\frac{\gamma\left(1+(p-2)\frac{\alpha}{ps}\right)}{Q}>\frac{2}{\tau Q}$, then
\begin{equation}\label{eq frac CKN in H^n i}
\left(\int_{\HH^n}d(\xi)^{\tau\gamma}|z|^{\tau\gamma(p-2)\frac{\alpha}{ps}}|f(\xi)|^{\tau}d\xi\right)^{\frac{1}{\tau}}\leq C [f]_{s,p,\alpha}^a\lVert f\rVert_{L^q(\HH^n)}^{1-a},\ \forall\, f\in C_c^1(\HH^n).
\end{equation}
\emph{(ii)} If $\frac{1}{\tau}+\frac{\gamma\left(1+(p-2)\frac{\alpha}{ps}\right)}{Q}<\frac{2}{\tau Q}$, then
\begin{equation}\label{eq frac CKN in H^n ii}
\left(\int_{\HH^n}d(\xi)^{\tau\gamma}|z|^{\tau\gamma(p-2)\frac{\alpha}{ps}}|f(\xi)|^{\tau}d\xi\right)^{\frac{1}{\tau}}\leq C [f]_{s,p,\alpha}^a\lVert f\rVert_{L^q(\HH^n)}^{1-a},\ \forall\, f\in C_c^1(\HH^n\setminus\{\xi=(z,t):z=0\}).
\end{equation}
\end{theorem}

Comparing with the corresponding result in the Euclidean case, it is natural to ask if \eqref{eq frac CKN in H^n i} holds for $0<1/\tau+\gamma\left(1+(p-2)\frac{\alpha}{ps}\right)/{Q}\leq\frac{2}{\tau Q}$ as well. 
In general, the answer is not known to us, but as our next result shows, under an additional conditions on indices $\tau,\gamma$ and $q$, \eqref{eq frac CKN in H^n i} continues to hold for $1/\tau+\gamma\left(1+(p-2)\frac{\alpha}{ps}\right)/{Q}>0$.
\begin{theorem}\label{frac CKN in H^n b}
Let $p$, $s$, $\alpha$, $\tau$, $a$, $\gamma$ and $q$ are as considered in Theorem \emph{\ref{frac CKN in H^n}} and satisfy \eqref{tau gamma rel} and \eqref{range of gamma}. If $\frac{1}{\tau}+\frac{\gamma\left(1+(p-2)\frac{\alpha}{ps}\right)}{Q}>0$ and $\tau\gamma\leq-2$, then for $q>\frac{(1-a)\tau Q}{Q-2+\tau\gamma(p-2)\frac{\alpha}{ps}}$ we have
\begin{equation}\label{eq frac CKN in H^n iii}
\left(\int_{\HH^n}d(\xi)^{\tau\gamma}|z|^{\tau\gamma(p-2)\frac{\alpha}{ps}}|f(\xi)|^{\tau}d\xi\right)^{\frac{1}{\tau}}\leq C [f]_{s,p,\alpha}^a\lVert f\rVert_{L^q(\HH^n)}^{1-a},\ \forall\, f\in C_c^1(\HH^n).
\end{equation}
\end{theorem}
When $1/\tau+\gamma\left(1+(p-2)\frac{\alpha}{ps}\right)/{Q}=0$, L.H.S of \eqref{eq frac CKN in H^n iii} need not be finite in general. As in the Euclidean case, correction by  a weight function leads to the following analogous inequalities in this case.
\begin{theorem}\label{theorem limiting case}
Let $p$, $s$, $\alpha$, $\tau$, $a$, and $\gamma$ be as in Theorem \emph{\ref{frac CKN in H^n}} and $q>\frac{(1-a)\tau Q}{Q-2+\tau\gamma(p-2)\frac{\alpha}{ps}}$ satisfy \eqref{tau gamma rel} and \eqref{range of gamma}. If
\begin{equation}\label{eq 1.15}
\frac{1}{\tau}+\frac{\gamma\left(1+(p-2)\frac{\alpha}{ps}\right)}{Q}=0,
\end{equation}
then\\
\emph{(i)} For $f\in C_c^1(\HH^n)$ with $\mathrm{supp}\,f\subset B_R(0)$, we have
\begin{equation}\label{eq limiting case}
\left(\int_{\HH^n}\frac{d(\xi)^{\tau\gamma}|z|^{\tau\gamma(p-2)\frac{\alpha}{ps}}}{\ln^{\tau}(4R/d(\xi))}|f(\xi)|^{\tau}d\xi\right)^{\frac{1}{\tau}}\leq C[f]_{s,p,\alpha}^a\lVert f\rVert_{L^q(\HH^n)}^{1-a}.
\end{equation}
\emph{(ii)} For $f\in C_c^1(\HH^n)$ with $\mathrm{supp}\,f\cap B_r(0)=\phi$, we have
\begin{equation}\label{eq limiting case b}
\left(\int_{\HH^n}\frac{d(\xi)^{\tau\gamma}|z|^{\tau\gamma(p-2)\frac{\alpha}{ps}}}{\ln^{\tau} (4d(\xi)/r)}|f(\xi)|^{\tau}d\xi\right)^{\frac{1}{\tau}}\leq C[f]_{s,p,\alpha}^a\lVert f\rVert_{L^q(\HH^n)}^{1-a}.
\end{equation}
\end{theorem}

For $a=1$ and $\gamma=-s,$ \eqref{eq frac CKN in H^n i} and \eqref{eq frac CKN in H^n iii} are the fractional Hardy inequality \eqref{eq frac-har in H^n} for the corresponding indices, and they hold  for $ps=2$ when $ps+(p-2)\alpha<Q$, and for $ps\leq2$ as well when $ps+(p-2)\alpha<Q-2$. This an improvement on the earlier results proved in \cite{Adi-Mallick,Suragan-Hardy}.\smallskip

The proof of our first result  follows the technique used in \cite{CKN} for the Euclidean spaces. We decompose $\HH^n$ into the sets
\begin{equation}\label{decom a1}
\mathscr{A}_k:=\{\xi=(z,t)\in \HH^n:2^k<|z|<2^{k+1}\}
\end{equation}
to deal with the term $|z|^{\tau\gamma(p-2)\frac{\alpha}{ps}}$ and then we decompose each $\mathscr{A}_k$ into
\begin{equation}\label{decom a2}
\mathscr{A}_{k,j}:=\{\xi=(z,t)\in\mathscr{A}_k:j2^{2k}<t<(j+1)2^{2k}\}
\end{equation}
so that $|\mathscr{A}_{k,j}|=C2^{kQ}$ (see Figure 1.1).
\begin{center}
\begin{tikzpicture}
\draw[thick,->] (-1,0) -- (6,0) node[anchor=north west] {t};
\draw[thick,->] (0,0) -- (0,4.5) node[anchor=south east] {z};
\draw[dashed] (-1,4) -- (0,4);
\draw (0,4) -- (4,4);
\draw[dashed] (4,4) -- (6,4);
\draw[dashed] (-1,2) -- (0,2);
\draw (0,2) -- (4,2);
\draw[dashed] (4,2) -- (6,2);
\draw[dashed] (-1,1) -- (0,1);
\draw (0,1) -- (4,1);
\draw[dashed] (4,1) -- (6,1);
\draw[dashed] (-1,0.5) -- (6,0.5);
\draw (4,1) -- (4,4);
\draw (3,1) -- (3,2);
\draw (2,1) -- (2,2);
\draw (1,1) -- (1,2);
\draw (-1,3) node{$\mathscr{A}_{k+1}$};
\draw (-1,1.5) node{$\mathscr{A}_{k}$};
\draw (-1,0.75) node{$\mathscr{A}_{k-1}$};
\draw (2,3) node{$\mathscr{A}_{k+1,0}$};
\draw (0.5,1.5) node{$\mathscr{A}_{k,0}$};
\draw (1.5,1.5) node{$\mathscr{A}_{k,1}$};
\draw (2.5,1.5) node{$\mathscr{A}_{k,2}$};
\draw (3.5,1.5) node{$\mathscr{A}_{k,3}$};
\end{tikzpicture}
\begin{tikzpicture}
\draw[thick,->] (-1,0) -- (6,0) node[anchor=north west] {t};
\draw[thick,->] (0,0) -- (0,4.5) node[anchor=south east] {z};
\draw[dashed] (-1,4) -- (6,4);
\draw[dashed] (-1,2) -- (6,2);
\draw[dashed] (-1,1) -- (6,1);
\draw[dashed] (-1,0.5) -- (6,0.5);
\draw(2,1) rectangle (3,2);
\draw[thick] (1,0.5) rectangle (4,4);
\draw (2.5,1.5) node{$\mathscr{A}_{k,j}$};
\draw (2.5,2.5) node{$\mathscr{B}_{k,j}$};
%\draw[red] (3,1) rectangle (4,2);
%\draw[red,thick] (2,0.5) rectangle (5,4);
%\draw (3.65,1.5) node{\textcolor{red}{$\mathscr{A}_{k,j+1}$}};
%\draw (4.3,2.5) node{\textcolor{red}{$2\mathscr{A}_{k,j+1}$}};
\end{tikzpicture}
\end{center}
$$\mathrm{Figure}\ 1.1\quad\qquad\qquad\qquad\qquad\qquad\qquad\qquad\mathrm{Figure}\ 1.2$$ \smallskip

For any measurable set $\Omega\subset\HH^n$ we denote by
$$(f)_\Omega:= \fint_\Omega f(\xi)d\xi=\frac{1}{|\Omega|}\int_\Omega f(\xi)d\xi ,$$
the average of the function $f$ over $\Omega$. If $\mathscr{A}$ is a bounded extension domain for $W^{s,p,\alpha}$ (discussed in next section), then the following inequality holds:
\begin{equation}\label{1.22}
\fint_{\mathscr{A}}|f-(f)_{\mathscr{A}}|^{\tau} d\xi\leq C[f]_{s,p,\alpha,\mathscr{A}}^{a\tau}\lVert f\rVert_{L^q(\mathscr{A})}^{(1-a)\tau},\ \ \forall f\in C(\Bar{\mathscr{A}}).
\end{equation}
Such an inequality plays a crucial role in the proof of the Euclidean results. It is not known to us whether the sets $\{\mathscr{A}_{k,j}\}$ are extension domains for $W^{s,p,\alpha}$ or not, but if $K$ is a compact subset of a bounded open set $\Omega$, the following inequality is always true:
\begin{equation}\label{1.23}
\fint_{K}|f-(f)_{\Omega}|^{\tau} d\xi\leq C(\Omega,K)[f]_{s,p,\alpha,\Omega}^{a\tau}\lVert f\rVert_{L^q(\Omega)}^{(1-a)\tau},\ \ \forall f\in C(\Bar{\Omega}).
\end{equation}
Here the constant $C(\Omega,K)$ depends on the sets $K$ and $\Omega$. In particular, we establish the following inequality (in Lemma \ref{lemma 3.2}):
\begin{equation}\label{1.25}
\fint_{\mathscr{A}_{k,j}}|f-(f)_{\mathscr{B}_{k,j}}|^{\tau} d\xi\leq C\,2^{-\frac{k}{\tau}\left(Q+\tau\gamma\left(1+(p-2)\frac{\alpha}{ps}\right)\right)}[f]_{s,p,\alpha,\mathscr{B}_{k,j}}^{a\tau}\lVert f\rVert_{L^q(\mathscr{B}_{k,j})}^{(1-a)\tau},
\end{equation}
where 
\begin{equation}\label{set 2A}
\mathscr{B}_{k,j}:=\{\xi=(z,t)\in\HH^n:2^{k-1}<|z|\leq 2^{k+2}, \ (j-1)2^{2k}<t<(j+2)2^{2k}\},
\end{equation}
(see Figure 1.2). Note that $\mathscr{A}_{k,j}\Subset\mathscr{B}_{k,j}$. By the notation $A\Subset B$ we mean that the closure of the set $A$ is a compact subset of $B$. The constant $C$ in \eqref{1.25} is independent of $k$ and $j$. The dependency of the constant $C(\Omega,K)$ in \eqref{1.23} on the sets is being taken care of by the factor $2^{-\frac{k}{\tau}\left(Q+\tau\gamma\left(1+(p-2)\frac{\alpha}{ps}\right)\right)}$.
\begin{remark}\label{remark 1}
\emph{Note that the sets $\mathscr{A}_{k,j}$'s are mutually disjoint and
$$\sum_k\sum_j\int_{\mathscr{A}_{k,j}}f(\xi)d\xi=\int_{\HH^n}f(\xi)d\xi$$
holds for any measurable function $f.$ This is not the case for $\mathscr{B}_{k,j}$ since these sets  are not mutually disjoint. However there is a positive constant $C$, independent of $k$ and $j$ such that
\begin{equation}
\sum_k\sum_j\int_{\mathscr{B}_{k,j}}f(\xi)d\xi\leq C\int_{\HH^n}f(\xi)d\xi
\end{equation}
holds for any non-negative function, which will take care the overlapping issue.}
\end{remark}
For the case $(p-2)\alpha=0,$ like the proof in the Euclidean case, it is sufficient to work with the annular  decomposition
\begin{equation}\label{decom b}
A_k:=A_{2^k,2^{k+1}}(0)=\{\xi\in\HH^n:2^k<d(\xi)\leq 2^{k+1}\},\ \ k \in \ZZ,
\end{equation}
of $\HH^n$ for $Q+\tau\gamma\left(1+(p-2)\frac{\alpha}{ps}\right)>0$. Whereas, for the case $(p-2)\alpha\neq0,$ we work with the decomposition \eqref{decom a2}  under the additional condition that $Q+\tau\gamma(1+(p-2)\frac{\alpha}{p s})>2$. In the proof of Theorem \ref{frac CKN in H^n b} and Theorem \ref{theorem limiting case} we use the decomposition \eqref{decom b} and we deal with the term $|z|^{\tau\gamma(p-2)\frac{\alpha}{ps}}$ by including this in the integral $\int_{\mathscr{B}_k}|z|^{\tau\gamma(p-2)\frac{\alpha}{ps}}|f(\xi)|^{\tau}d\xi$ (see Lemma \ref{lemma 3.7}), where the condition $\tau\gamma\leq-2$ is crucial.\smallskip

For related work on domains, Orlicz spaces and other generalisations we refer to \cite{  subha3,subha2,roy2022,lu2004, lu2013,lu2009}. Recently several geometric Hardy and Caffarelli-Kohn-Nirenberg identities on general domains, Carnot groups and on Cartan-Hadamard manifolds has been established in \cite{luckn,GuozhenLu2,lucartan, GuozhenLu1}. Hardy identities are important in establishing more precise and stronger inequalities than known Hardy’s inequalities in the literature.\smallskip

The paper is organized as follows: In Section 2 we discuss some basic properties and embedding theorems of the space $W^{s,p,\alpha}(\Omega)$. Section 3 contains the proof of Theorem \ref{frac CKN in H^n}. In Section 4 we first establish some relevant results (Proposition \ref{prop 4.1} and Lemma \ref{lemma 3.7}), and then with the help of those results we prove Theorem \ref{frac CKN in H^n b}. We conclude the paper with the proof of Theorem \ref{theorem limiting case} in Section 5.

\section{Some basic properties of fractional Sobolev spaces}
The following result gives the relation between the spaces $W^{s,p,\alpha}(\Omega)$ and $ W^{s',p,\alpha}(\Omega)$ for $s'<s$. This result for the case $(p-2)\alpha=0$ is already established in \cite{Maria}.
\begin{proposition}\label{prop 2.1}
Let $1\leq p<\infty$ and $\alpha\in\RR$, $0<s'\leq s<1$ and $\Omega$ be an open set in $\HH^n$. Then for any measurable function $f:\Omega\rightarrow\RR$,
\begin{equation}\label{2.1}
\lVert f\rVert_{s',p,\alpha,\Omega}\leq C \lVert f\rVert_{s,p,\alpha,\Omega}
\end{equation}
for some positive constant $C=C(n,p,s,\alpha)$. In other word, $W^{s,p,\alpha}(\Omega)\subseteq W^{s',p,\alpha}(\Omega)$.
\end{proposition}
\begin{proof}
We have
\begin{multline}\label{2.2}
[f]_{s',p,\alpha,\Omega}^p=\int_{\Omega} \int_{\Omega\cap D_1(\xi)}\frac{|f(\xi)-f(\xi')|^pd\xi' d\xi}{d({\xi}^{-1}\circ \xi')^{Q+ps'} |z-z'|^{(p-2)\alpha}}\\
+\int_{\Omega} \int_{\Omega\cap D_1(\xi)^C}\frac{|f(\xi)-f(\xi')|^pd\xi' d\xi}{d({\xi}^{-1}\circ \xi')^{Q+ps'} |z-z'|^{(p-2)\alpha}}=I_1+I_2,
\end{multline}
where $D_1(\xi)$ is the set defined in \eqref{set D_r}. Since $d({\xi}^{-1}\circ \xi')<1$ in $D_1(\xi)$,
\begin{equation}\label{2.3}
I_1\leq \int_{\Omega} \int_{\Omega\cap D_1(\xi)}\frac{|f(\xi)-f(\xi')|^pd\xi' d\xi}{d({\xi}^{-1}\circ \xi')^{Q+ps} |z-z'|^{(p-2)\alpha}}\leq [f]_{s,p,\alpha,\Omega}^p.
\end{equation}
We now estimate the integral $I_2$.
$$I_2\leq 2^{p-1}\int_{\Omega} \int_{\Omega\cap D_1(\xi)^C}\frac{(|f(\xi)|^p+|f(\xi')|^p)d\xi' d\xi}{d({\xi}^{-1}\circ \xi')^{Q+ps'} |z-z'|^{(p-2)\alpha}}.$$
Using symmetry, we obtain
\begin{equation}\label{2.4}
I_2\leq 2^p\int_{\Omega}\left(\int_{\Omega\cap D_1(\xi)^C}\frac{d\xi'}{d({\xi}^{-1}\circ \xi')^{Q+ps'} |z-z'|^{(p-2)\alpha}}\right)|f(\xi)|^pd\xi.
\end{equation}
Let
\begin{multline*}
J=\int_{\Omega\cap D_1(\xi)^C}\frac{d\xi'}{d({\xi}^{-1}\circ \xi')^{Q+ps'} |z-z'|^{(p-2)\alpha}}\\
\leq \int_{D_1(\xi)^C}\frac{dz'dt'}{(|z-z'|^4+(t'-t-2\langle y,x'\rangle+2\langle x,y'\rangle)^2)^{\frac{Q+ps'}{4}}|z-z'|^{(p-2)\alpha}}.
\end{multline*}
We use the change of variable $\zeta=z'-z$ and $\mu=t'-t-2\langle y,x'\rangle+2\langle x,y'\rangle$, and obtain
$$J\leq \int_{D_1(0)^C}\frac{d\zeta d\mu}{(|\zeta|^4+\mu^2)^{\frac{Q+ps'}{4}}|\zeta|^{(p-2)\alpha}}.$$
Let us define the following sets:
$$B_1=\{(z,t)\in \RR^{2n}\times\RR:|z|>1,\, |t|>1\},$$
$$B_2=\{(z,t)\in \RR^{2n}\times\RR:|z|>1,\, |t|<1\},$$
$$B_3=\{(z,t)\in \RR^{2n}\times\RR:|z|<1,\, |t|>1\},$$
and the integrals
$$J_1= \int_{B_1}\frac{d\zeta d\mu}{(|\zeta|^4+\mu^2)^{\frac{Q+ps'}{4}}|\zeta|^{(p-2)\alpha}},$$
$$J_2= \int_{B_2}\frac{d\zeta d\mu}{(|\zeta|^4+\mu^2)^{\frac{Q+ps'}{4}}|\zeta|^{(p-2)\alpha}},$$
$$J_3= \int_{B_3}\frac{d\zeta d\mu}{(|\zeta|^4+\mu^2)^{\frac{Q+ps'}{4}}|\zeta|^{(p-2)\alpha}}.$$
Clearly $D_1(0)^C=B_1\cup B_2\cup B_3$ up to a measure zero set, and $J\leq J_1+J_2+J_3$.\\
\textbf{Estimate for $J_1$:} Let $\omega_{2n}$ be the measure of the unit sphere in $\RR^{2n}$.
\begin{multline*}
J_1=2\omega_{2n}\int_1^{\infty}\int_1^{\infty}\frac{r^{Q-3}dr d\mu}{(r^4+\mu^2)^{\frac{Q+ps'}{4}}r^{(p-2)\alpha}}\\
(\mathrm{using}\ \mu=r^2t)\ =2\omega_{2n}\int_1^{\infty}\left(\int_{\frac{1}{r^2}}^\infty\frac{dt}{(1+t^2)^{\frac{Q+ps'}{4}}}\right)\frac{dr}{r^{1+ps'+(p-2)\alpha}}\\
\leq 2\omega_{2n}\int_1^{\infty}\left(\int_0^\infty\frac{dt}{(1+t^2)^{\frac{Q+ps'}{4}}}\right)\frac{dr}{r^{1+ps'+(p-2)\alpha}}.
\end{multline*}
Since $Q+ps'>2$ and $ps'+(p-2)\alpha>0$, both the integrals of R.H.S is finite. Hence $J_1\leq C_1$.\\
\textbf{Estimate for $J_2$:}
\begin{multline*}
J_2=2\omega_{2n}\int_1^{\infty}\int_0^1\frac{r^{Q-3}dr d\mu}{(r^4+\mu^2)^{\frac{Q+ps'}{4}}r^{(p-2)\alpha}}\\
(\mathrm{using}\ \mu=r^2t)\ =2\omega_{2n}\int_1^{\infty}\left(\int_0^{\frac{1}{r^2}}\frac{dt}{(1+t^2)^{\frac{Q+ps'}{4}}}\right)\frac{dr}{r^{1+ps'+(p-2)\alpha}}\\
\leq 2\omega_{2n}\int_1^{\infty}\left(\int_0^1\frac{dt}{(1+t^2)^{\frac{Q+ps'}{4}}}\right)\frac{dr}{r^{1+ps'+(p-2)\alpha}}\leq C_2.
\end{multline*}
\textbf{Estimate for $J_3$:}
$$J_3\leq 2\omega_{2n}\int_0^1\left(\int_1^\infty\frac{d\mu}{\mu^{\frac{Q+ps'}{2}}}\right)r^{Q-3-(p-2)\alpha}dr.$$
Since $\frac{Q+ps'}{2}>1$ and $(p-2)\alpha<Q-2$, we have $J_3\leq C_3$.
Hence $J$ is integrable.\\
Hence from \eqref{2.4}, we get
\begin{equation}\label{2.5}
I_2\leq 2^pC\lVert f\rVert_{L^p(\Omega)}^p.
\end{equation}
Combining \eqref{2.2}, \eqref{2.3} and \eqref{2.5}, we obtain
\begin{equation}
\lVert f\rVert_{s',p,\alpha,\Omega}^p\leq (2^pC+1)\lVert f\rVert_{L^p(\Omega)}^p+[f]_{s,p,\alpha,\Omega}^p\leq C(n,p,s,\alpha) \lVert f\rVert_{s,p,\alpha,\Omega}^p.
\end{equation}
\end{proof}
\begin{definition}
Let $1\leq p<\infty$, $s\in(0,1)$ and $\alpha\in\RR$, and let $\Omega\subset\HH^n$ be an open set. We say that $\Omega$ is an extension domain, if there exists a positive constant $C=C(n,p,s,\alpha,\Omega)$ such that, for any function $f\in W^{s,p,\alpha}(\Omega)$, there exists $\Bar{f}\in W_0^{s,p,\alpha}(\HH^n)$ with $\Bar{f}|_\Omega=f$, and
\begin{equation}\label{extension}
\lVert \Bar{f}\rVert_{s,p,\alpha}\leq C\lVert f\rVert_{s,p,\alpha,\Omega}.
\end{equation}
\end{definition}
In the following result we extend any function of $W^{s,p,\alpha}(\Omega)$, supported in $K\Subset\Omega$ to a function in $W^{s,p,\alpha}(\HH^n)$.
\begin{proposition}\label{prop 2.2}
Let $1\leq p<\infty$, $s\in(0,1)$, $\alpha\in\RR$ and $\Omega$ be any open set in $\HH^n$. If $f\in W^{s,p,\alpha}(\Omega)$ and $K$ is any compact subset of $\Omega$ such that $f\equiv0$ in $\Omega\setminus K$, then the extension function
$$\Bar{f}(\xi)=\begin{cases}
f(\xi) & \text{if $\xi\in\Omega$}\\
0 & \text{if $\xi\in\HH^n\setminus\Omega$}
\end{cases}$$
belongs to $W^{s,p,\alpha}(\HH^n)$, and there exists a positive constant $C$ depending on $n$, $p$, $s$, $\alpha$, $K$, $\Omega$ and is such that
\begin{equation}\label{eq 2.8}
\lVert \Bar{f}\rVert_{s,p,\alpha}\leq C\lVert f\rVert_{s,p,\alpha,\Omega}.
\end{equation}
\end{proposition}
\begin{proof}
We have
\begin{multline}
[\Bar{f}]_{s,p,\alpha}^p=\int_{\mathbb{H}^n} \int_{\mathbb{H}^n}\frac{|\Bar{f}(\xi)-\Bar{f}(\xi')|^pd\xi' d\xi}{d({\xi}^{-1}\circ \xi')^{Q+ps} |z-z'|^{(p-2)\alpha}}\\
=\int_{\Omega} \int_{\Omega}\frac{|f(\xi)-f(\xi')|^pd\xi' d\xi}{d({\xi}^{-1}\circ \xi')^{Q+ps} |z-z'|^{(p-2)\alpha}}+2\int_{K} \int_{\Omega^C}\frac{|f(\xi)|^pd\xi' d\xi}{d({\xi}^{-1}\circ \xi')^{Q+ps} |z-z'|^{(p-2)\alpha}}\\
=[f]^p_{s,p,\alpha,\Omega}+2\int_{K} |f(\xi)|^p\int_{\Omega^C}\frac{d\xi'}{d({\xi}^{-1}\circ \xi')^{Q+ps} |z-z'|^{(p-2)\alpha}} d\xi.
\end{multline}
We claim that $\exists\,\beta>0$ such that $d({\xi}^{-1}\circ \xi')\geq\beta$ for all $\xi\in K$ and $\xi'\in\Omega^C$. If not, then $\exists\,\xi\in K$ and $\{\xi_m'\}\subset\Omega^C$ such that $d({\xi}^{-1}\circ \xi_m')<\frac{1}{m}$ $\forall m\in\NN$. In particular, $|z-z_m'|<\frac{1}{m}$ and $|t_m'-t-2\langle y,x_m'\rangle+2\langle x,y_m'\rangle|<\frac{1}{m^2}$, where $\xi_m'=(x_m',y_m',t_m')=(z_m',t_m')$. Now
\begin{multline*}
|t_m'-t|\leq |t_m'-t-2\langle y,x_m'\rangle+2\langle x,y_m'\rangle|+2|\langle y,x_m'\rangle-\langle x,y_m'\rangle|\\
=|t_m'-t-2\langle y,x_m'\rangle+2\langle x,y_m'\rangle|+2|\langle(-y,x),z-z_m'\rangle|\\
\leq |t_m'-t-2\langle y,x_m'\rangle+2\langle x,y_m'\rangle|+2|z||z-z_m'|\leq \frac{1}{m^2}+\frac{2R}{m},
\end{multline*}
where $R>0$ is such that $K\subset B(0,R)$, which implies that the sequence $\{\xi_m'\}$ converges to $\xi$ and contradicts the fact $K$ is a compact subset of $\Omega$.\\
Hence we have for any $\xi\in K$,
\begin{multline*}
\int_{\Omega^C}\frac{d\xi'}{d({\xi}^{-1}\circ \xi')^{Q+ps} |z-z'|^{(p-2)\alpha}}\\
\leq\int_{B_\beta(\xi)^C}\frac{d\xi'}{d({\xi}^{-1}\circ \xi')^{Q+ps} |z-z'|^{(p-2)\alpha}}\leq \int_{B_\beta(0)^C}\frac{d\xi'}{d(\xi')^{Q+ps} |z'|^{(p-2)\alpha}}.
\end{multline*}
One can show the last integral is finite by an argument similar to one  as in Proposition \ref{prop 2.1}.
This gives
\begin{equation}
[\Bar{f}]_{s,p,\alpha}^p\leq[f]^p_{s,p,\alpha,\Omega}+C\lVert f\rVert_{L^p(\Omega)}^p.
\end{equation}
This completes the proof.
\end{proof}
\begin{corollary}\label{coro 1}
Let $1\leq p<\infty$, $s\in(0,1)$ and $\alpha\in\RR$ be such that $ps+(p-2)\alpha<Q$. Let $\Omega\subset\HH^n$ be an open set and $K$ be any compact subset of $\Omega$. Then there exists a positive constant $C$, which depends on $n$, $p$, $s$, $\alpha$ ,$K,$ $\Omega$ and is such that
\begin{equation}\label{eq frac-sob in H^n 1}
\lVert f\rVert_{L^{p^*}(K)}\leq C\lVert f\rVert_{s,p,\alpha,\Omega},\  \forall f\in C(\bar{\Omega}),
\end{equation}
where $p^*=\frac{Qp}{Q-ps-(p-2)\alpha}$.
\end{corollary}
\begin{proof}
Let $K'$ be an open set such that $K\Subset K'\Subset\Omega$. Choose a function $\phi\in C_c^\infty(\HH^n)$ such that $0\leq\phi\leq1$, $\phi\equiv1$ in $K$ and $\phi\equiv0$ in $\HH^n\setminus K'$.\\
Clearly $f\phi\in C_c(\HH^n)$. Sobolev inequality \eqref{eq frac-sob in H^n} gives,
\begin{equation}
\left(\int_{K}|f|^{p^*}\right)^{\frac{1}{p^*}}\leq \left(\int_{\HH^n}|f\phi|^{p^*}\right)^{\frac{1}{p^*}}\leq C[f\phi]_{s,p,\alpha}.
\end{equation}
It follows from \eqref{eq 2.8},
\begin{equation}\label{eq 2.13}
\lVert f\rVert_{L^{p^*}(K)}\leq C\lVert f\phi\rVert_{s,p,\alpha,\Omega}\leq C(\lVert f\rVert_{L^{p}(\Omega)}+[f\phi]_{s,p,\alpha,\Omega}).
\end{equation}
We now estimate the quantity $[f\phi]_{s,p,\alpha,\Omega}$ in terms of $\lVert f\rVert_{s,p,\alpha,\Omega}$.
\begin{multline}
[f\phi]^p_{s,p,\alpha,\Omega}\leq2^{p-1}\int_{\Omega} |f(\xi)|^p\int_{\Omega}\frac{|\phi(\xi)-\phi(\xi')|^pd\xi' d\xi}{d({\xi}^{-1}\circ \xi')^{Q+ps} |z-z'|^{(p-2)\alpha}}\\
+2^{p-1}\int_{\Omega} \int_{\Omega}\frac{|f(\xi)-f(\xi')|^p|\phi(\xi')|^pd\xi' d\xi}{d({\xi}^{-1}\circ \xi')^{Q+ps} |z-z'|^{(p-2)\alpha}}\\
\leq 2^{2p-1}\int_{\Omega} |f(\xi)|^p\int_{B(\xi,1)}\frac{|\phi(\xi)-\phi(\xi')|^pd\xi' d\xi}{d({\xi}^{-1}\circ \xi')^{Q+ps} |z-z'|^{(p-2)\alpha}}\\
+2^{2p-1}\int_{\Omega} |f(\xi)|^p\int_{B(\xi,1)^C}\frac{d\xi' }{d({\xi}^{-1}\circ \xi')^{Q+ps} |z-z'|^{(p-2)\alpha}}d\xi+2^{p-1}[f]^p_{s,p,\alpha,\Omega}.
\end{multline}
In the last integral we use $\phi\leq1$.
Since $\phi\in C_c^\infty(\HH^n)$, $\exists M>0$ such that $\frac{|\phi(\xi)-\phi(\xi')|}{d({\xi}^{-1}\circ \xi')}\leq M$ for all $\xi,\xi'\in \Omega$,  (see \cite{Adi-Mallick}, Proposition 3.1). Hence for any $\xi\in\Omega$,
$$\int_{B(\xi,1)}\frac{|\phi(\xi)-\phi(\xi')|^pd\xi'}{d({\xi}^{-1}\circ \xi')^{Q+ps} |z-z'|^{(p-2)\alpha}}\leq M^p\int_{B(\xi,1)}\frac{d\xi'}{d({\xi}^{-1}\circ \xi')^{Q+ps-p} |z-z'|^{(p-2)\alpha}}.$$
Both the integrals
$$\int_{B(\xi,1)}\frac{d\xi'}{d({\xi}^{-1}\circ \xi')^{Q+ps-p} |z-z'|^{(p-2)\alpha}}d\xi$$
and
$$\int_{B(\xi,1)^C}\frac{d\xi' }{d({\xi}^{-1}\circ \xi')^{Q+ps} |z-z'|^{(p-2)\alpha}}$$
are shown to be finite in \cite{Adi-Mallick}, (Proposition 3.1). Hence
\begin{equation}\label{3.4}
[f\phi]^p_{s,p,\alpha,\Omega}\leq C\left(\lVert f\rVert_{L^p(\Omega)}^p+[f]^p_{s,p,\alpha,\Omega}\right)\leq C\lVert f\rVert_{s,p,\alpha,\Omega}^p.
\end{equation}
Inequality \eqref{eq frac-sob in H^n 1} follows from \eqref{eq 2.13} and \eqref{3.4}.
\end{proof}
\begin{theorem}\label{theorem 2.2}
Let $1\leq p<\infty$, $s\in(0,1)$ and $\alpha\in\RR$ be such $ps+(p-2)\alpha<Q$. Let $\Omega\subset\HH^n$ be an extension domain. Then there exists a positive constant $C=C(n,p,s,\alpha,\Omega)$ such that for any $q\in[p,p^*]$ and for any $f\in W^{s,p,\alpha}(\Omega)$, we have
\begin{equation}\label{2.8}
\lVert f\rVert_{L^q(\Omega)}\leq C\lVert f\rVert_{s,p,\alpha,\Omega}.
\end{equation}
In other word, $W^{s,p,\alpha}(\Omega)$ is continuously embedded in $L^q(\Omega)$ for any $q\in[p,p^*]$.
\end{theorem}
\begin{proof}
The case $q=p$ and $q=p^*$ is trivial. For the case $q\in(p,p^*)$, choose $0<\theta<1$ such that $\frac{1}{q}=\frac{\theta}{p}+\frac{1-\theta}{p^*}$. Inequality \eqref{2.8} follows from H\"older's inequality with the conjugets $\frac{p}{\theta q}$ and $\frac{p^*}{(1-\theta)q}$, \eqref{eq frac-sob in H^n} and \eqref{extension}.
\end{proof}
\begin{remark}
When $ps+(p-2)\alpha\rightarrow Q$, clearly the fractional critical Sobolev exponent $p^*$ goes to $\infty$. Hence if $ps+(p-2)\alpha=Q$, one may expect that $W^{s,p,\alpha}(\Omega)$ is continuously embedded in $L^q(\Omega)$ for any $q\geq p$. Our next result assures that this embedding is true even in the case $ps+(p-2)\alpha\geq Q$.
\end{remark}
\begin{theorem}
Let $1\leq p<\infty$, $s\in(0,1)$ and $\alpha\in\RR$ be such $ps+(p-2)\alpha\geq Q$. Let $\Omega\subset\HH^n$ be an extension domain. Then there exists a positive constant $C=C(n,p,s,\alpha,\Omega)$ such that for any $q\geq p$ and for any $f\in W^{s,p,\alpha}(\Omega)$, we have
\begin{equation}\label{2.9}
\lVert f\rVert_{L^q(\Omega)}\leq C\lVert f\rVert_{s,p,\alpha,\Omega}.
\end{equation}
\end{theorem}
\begin{proof}
Since $(p-2)\alpha<Q-2$, we can choose $0<s'<s$ such that $ps'+(p-2)\alpha<Q$ and $q\leq\frac{pQ}{Q-ps'-(p-2)\alpha}$. By Theorem \ref{theorem 2.2}, we have
\begin{equation}
\lVert f\rVert_{L^q(\Omega)}\leq C\lVert f\rVert_{s',p,\alpha,\Omega}.
\end{equation}
Inequality \eqref{2.9} then follows from Proposition \ref{prop 2.1}.
\end{proof}

\section{Proof of Theorem \ref{frac CKN in H^n}}
The following lemma, which uses Corollary \ref{coro 1} is key to the proof of Theorem \ref{frac CKN in H^n}.

\begin{lemma}\label{lemma 3.2}
Let $p>1$, $s\in(0,1)$, $\alpha\in\RR$, $\tau>1$,$\,\gamma\in\RR$, $0<a\leq1$ and $q>(1-a)\tau$ satisfy \eqref{tau gamma rel} and \eqref{range of gamma}. Let $\lambda>0$, $0<r'<r<R<R'$, $\mu,\mu' \in\RR$ and $\ell,\ell'>0$ be such that $\mu'<\mu$ and $\mu+\ell<\mu'+\ell'$ and consider the following sets:
$$\mathscr{D}_\lambda=\mathscr{D}_{r,R,\mu,\ell}(\lambda):=\{\xi=(z,t)\in\HH^n:\lambda r<|z|<\lambda R\ \mathrm{and}\ \ \mu\lambda^2<t<(\mu+\ell)\lambda^2\}$$
and $\mathscr{D}_\lambda'=\mathscr{D}_{r',R',\mu',\ell'}(\lambda)$. Then for $f\in C(\Bar{\mathscr{D}}_{\lambda}')$, we have
\begin{equation}\label{eq 3.3}
\left(\fint_{\mathscr{D}_\lambda}|f-(f)_{\mathscr{D}_\lambda'}|^{\tau}\right)^{\frac{1}{\tau}}\leq C\left(\lambda^{ps+(p-2)\alpha-Q}[f]_{s,p,\alpha,{\mathscr{D}_\lambda'}}^p\right)^{\frac{a}{p}}\left(\fint_{\mathscr{D}_\lambda'}|f|^{q} d\xi\right)^{\frac{1-a}{q}},
\end{equation}
where  the constant $C>0$ is independent of $f$ and $\lambda$, $\mu$ and $\mu'$.
\end{lemma}
\begin{proof}
From \eqref{tau gamma rel}, we have
$$\frac{1}{\tau}=a\left(\frac{1}{p}-\frac{s'\left(1+(p-2)\frac{\alpha}{ps}\right)}{Q}\right)+\frac{1-a}{q},$$
where $s'=s+\gamma/a<s$. Let $\alpha'=\alpha s'/s$ and ${p'}^*=\frac{pQ}{Q-ps'-(p-2)\alpha'}$. We have $1/\tau=a/{p'}^*+(1-a)/q$ and ${p'}^*>0$ follows from the range of $q$.\\
For any function $f$, we define the function $\Tilde{f}(\xi)=f(D_\lambda \xi)$, where $D_\lambda \xi=(\lambda z,\lambda^2t)$, is the dilation in $\HH^n$. Let us define the sets
$$\mathscr{D}_1=\{\xi=(z,t)\in\HH^n: r<|z|\leq  R\ \mathrm{and}\ \ \mu<t\leq \mu+\ell\},$$
$$\mathscr{D}_1'=\{\xi=(z,t)\in\HH^n: r'<|z|\leq R'\ \mathrm{and}\ \ \mu'<t\leq \mu'+\ell'\}.$$
Clearly, $f\in C^1(\Bar{\mathscr{D}}_\lambda')$ if and only if $\Tilde{f}\in C^1(\Bar{\mathscr{D}}_1')$. H\"older's inequality, inequality \eqref{eq frac-sob in H^n 1} with $K=\Bar{\mathscr{D}}_1$ and $\Omega=\mathscr{D}_1'$, and the fact $[\Tilde{f}]_{s',p,\alpha',\mathscr{D}_1'}\leq C[\Tilde{f}]_{s,p,\alpha,\mathscr{D}_1'}$ respectively gives
\begin{multline}\label{eq 3.4}
\left(\int_{\mathscr{D}_1}|\Tilde{f}-(f)_{\mathscr{D}_\lambda'}|^{\tau} \right)^{\frac{1}{\tau}}\leq \left(\int_{\mathscr{D}_1}|\Tilde{f}-(f)_{\mathscr{D}_\lambda'}|^{{p'}^*} \right)^{\frac{a}{{p'}^*}}\left(\int_{\mathscr{D}_1}|\Tilde{f}-(f)_{\mathscr{D}_\lambda'}|^{q} \right)^{\frac{1-a}{q}}\\
\leq C\left(\lVert \Tilde{f}-(f)_{\mathscr{D}_\lambda'}\rVert_{L^p(\mathscr{D}_1')}^p+[\Tilde{f}]_{s',p,\alpha',\mathscr{D}_1'}^p\right)^{\frac{a}{p}}\left(\int_{\mathscr{D}_1}|\Tilde{f}-(f)_{\mathscr{D}_\lambda'}|^{q} \right)^{\frac{1-a}{q}}\\
\leq C\left(\lVert \Tilde{f}-(f)_{\mathscr{D}_\lambda'}\rVert_{L^p(\mathscr{D}_1')}^p+[\Tilde{f}]_{s,p,\alpha,\mathscr{D}_1'}^p\right)^{\frac{a}{p}}\left(\int_{\mathscr{D}_1}|\Tilde{f}-(f)_{\mathscr{D}_\lambda'}|^{q} \right)^{\frac{1-a}{q}}.
\end{multline}
Using the change of variable $\Bar{\xi}=D_\lambda \xi$ in \eqref{eq 3.4} and obtain
\begin{multline}\label{eq 3.5}
\left(\fint_{\mathscr{D}_\lambda}|f-(f)_{\mathscr{D}_\lambda'}|^{\tau}\right)^{\frac{1}{\tau}}\\
\leq C\left(\fint_{\mathscr{D}_\lambda'}|f-(f)_{\mathscr{D}_\lambda'}|^{p}+\lambda^{ps+(p-2)\alpha-Q}[f]^p_{s,p,\alpha,\mathscr{D}_\lambda'}\right)^{\frac{a}{p}}\left(\fint_{\mathscr{D}_\lambda}|f-(f)_{\mathscr{D}_\lambda'}|^{q}\right)^{\frac{1-a}{q}}\\
\leq C\left(\fint_{\mathscr{D}_\lambda'}|f-(f)_{\mathscr{D}_\lambda'}|^{p}+\lambda^{ps+(p-2)\alpha-Q}[f]^p_{s,p,\alpha,\mathscr{D}_\lambda'}\right)^{\frac{a}{p}}\left(\fint_{\mathscr{D}_\lambda'}|f|^{q}\right)^{\frac{1-a}{q}}.
\end{multline}
Now
\begin{multline}\label{eq 3.6}
\int_{\mathscr{D}_\lambda'}|f(\xi)-(f)_{\mathscr{D}_\lambda'}|^{p}d\xi=\int_{\mathscr{D}_\lambda'}\left|f(\xi)-\frac{1}{|{\mathscr{D}_\lambda'}|}\int_{\mathscr{D}_\lambda'}f(\xi')d\xi'\right|^{p}d\xi\\
=\int_{\mathscr{D}_\lambda'}\left|\frac{1}{|{\mathscr{D}_\lambda'}|}\int_{\mathscr{D}_\lambda'}\left(f(\xi)-f(\xi')\right)d\xi'\right|^{p}d\xi\leq \int_{\mathscr{D}_\lambda'}\left(\fint_{\mathscr{D}_\lambda'}\left|f(\xi)-f(\xi')\right|d\xi'\right)^{p}d\xi.
\end{multline}
Using Jensen's inequality, we get
\begin{equation}\label{eq 3.7}
\left(\fint_{\mathscr{D}_\lambda'}\left|f(\xi)-f(\xi')\right|d\xi'\right)^{p}\leq\frac{1}{|{\mathscr{D}_\lambda'}|}\int_{\mathscr{D}_\lambda'}\left|f(\xi)-f(\xi')\right|^pd\xi'.
\end{equation}
Let $\xi,\xi'\in \mathscr{D}_\lambda'$, where $\xi=(z,t)=(x,y,t)$ and $\xi'=(z',t')=(x',y',t')$. We have
\begin{multline*}
d({\xi}^{-1}\circ \xi')=\left(|z'-z|^4+(t'-t+2(\langle x,y'\rangle-\langle y,x'\rangle))^2\right)^{\frac{1}{4}}\\
\leq \left(|z'-z|^4+2(t'-t)^2+8(\langle x,y'\rangle-\langle y,x'\rangle)^2\right)^{\frac{1}{4}}\\
=\left(|z'-z|^4+2(t'-t)^2+8(\langle (y',-x'),z-z'\rangle)^2\right)^{\frac{1}{4}}.
\end{multline*}
Cauchy-Schwarz inequality in $\RR^{2n}$ gives
\begin{multline*}
d({\xi}^{-1}\circ \xi')\leq \left(|z'-z|^4+2(t'-t)^2+8|z'|^2|z-z'|^2\right)^{\frac{1}{4}}\\
\leq((2{R'})^4\lambda^4+2\cdot{\ell'}^2\lambda^4+8{R'}^2(2{R'})^2\lambda^4)^{\frac{1}{4}}=C\lambda.
\end{multline*}
Note that $|{\mathscr{D}_\lambda'}|=C\lambda^Q$. Hence, from \eqref{eq 3.6} and \eqref{eq 3.7}, we get
\begin{multline}\label{eq 3.8}
\int_{\mathscr{D}_\lambda'}|f(\xi)-(f)_{\mathscr{D}_\lambda'}|^{p}d\xi\\
\leq C\lambda^{ps+(p-2)\alpha}\int_{\mathscr{D}_\lambda'}\int_{\mathscr{D}_\lambda'}\frac{|f(\xi)-f(\xi')|^p}{d({\xi}^{-1}\circ \xi')^{Q+ps} |z-z'|^{(p-2)\alpha}}d\xi=C\lambda^{ps+(p-2)\alpha}[f]_{s,p,\alpha,{\mathscr{D}_\lambda'}}^p,
\end{multline}
which gives
\begin{equation}\label{eq 3.9}
\fint_{\mathscr{D}_\lambda'}|f(\xi)-(f)_{\mathscr{D}_\lambda'}|^{p}d\xi\leq C\lambda^{ps+(p-2)\alpha-Q} [f]_{s,p,\alpha,{\mathscr{D}_\lambda'}}^p.
\end{equation}
Inequality \eqref{eq 3.3} follows from \eqref{eq 3.5} and \eqref{eq 3.9}.
\end{proof}
\begin{remark}\label{remark 4}
\emph{For $i=0,1,2$ and $3$ we consider the connected domain $\mathscr{B}_{k+1,j}\cup\mathscr{B}_{k,4j+i}$ as defined in \eqref{set 2A}, and we define the set
$$\mathscr{C}_{k,j}:=\{\xi=(z,t)\in\HH^n:2^{k-2}<|z|\leq 2^{k+3}, \ (j-2)2^{2k}<t<(j+3)2^{2k}\}.$$
Observe that $\mathscr{B}_{k+1,j}\cup\mathscr{B}_{k,4j+i}\Subset \mathscr{C}_{k+1,j}\cup\mathscr{C}_{k,4j+i}$ and the following inequality is true:
\begin{multline}\label{3.8}
\left(\fint_{\mathscr{B}_{k+1,j}\cup\mathscr{B}_{k,4j+i}}|f-(f)_{\mathscr{C}_{k+1,j}\cup\mathscr{C}_{k,4j+i}}|^{\tau}\right)^{\frac{1}{\tau}}\\
\leq C\left(2^{k(ps+(p-2)\alpha-Q)}[f]_{s,p,\alpha,\mathscr{C}_{k+1,j}\cup\mathscr{C}_{k,4j+i}}^p\right)^{\frac{a}{p}}\left(\fint_{\mathscr{C}_{k+1,j}\cup\mathscr{C}_{k,4j+i}}|f|^{q} d\xi\right)^{\frac{1-a}{q}},
\end{multline}
where the constant $C$ is independent of $k$ and $j$.}
\end{remark}
The following elementary lemma is established in \cite{Vivek}.
\begin{lemma}\label{lemma 4.1}
Let $\tau>1$ and $c>1$. Then for all $a,b\in\RR$, we have
\begin{equation}\label{eq lemma 4.1}
(|a|+|b|)^\tau\leq c|a|^\tau+\left(1-c^{-\frac{1}{\tau-1}}\right)^{1-\tau}|b|^\tau.
\end{equation}
\end{lemma}

\begin{proof}[\textbf{Proof of Theorem} \ref{frac CKN in H^n}]
\textbf{Step 1:} 
Proof of (i). Let $\mathrm{supp}f\subset D_{2^{n_0+1}}(0)$ for some $n_0\in\ZZ$. For $k, j\in\ZZ$, we consider the sets $\{\mathscr{A}_k\}$, $\{\mathscr{A}_{k,j}\}$ and $\mathscr{B}_{k,j}$ defined in \eqref{decom a1}, \eqref{decom a2} and \eqref{set 2A} respectively.\\
By Lemma \ref{lemma 3.2} and by the relation \eqref{tau gamma rel} we have
\begin{multline}\label{eq 3.10}
\left(\fint_{\mathscr{A}_{k,j}}|f-(f)_{\mathscr{B}_{k,j}}|^{\tau} d\xi\right)^{\frac{1}{\tau}}\leq 
C\left(2^{k(ps+(p-2)\alpha-Q)}[f]_{s,p,\alpha,\mathscr{B}_{k,j}}^p\right)^{\frac{a}{p}}\left(\frac{1}{2^{kQ}}\int_{\mathscr{B}_{k,j}}|f|^{q} d\xi\right)^{\frac{1-a}{q}}\\
=C\,2^{-\frac{k}{\tau}\left(Q+\tau\gamma\left(1+(p-2)\frac{\alpha}{ps}\right)\right)}[f]_{s,p,\alpha,\mathscr{B}_{k,j}}^a\lVert f\rVert_{L^q(\mathscr{B}_{k,j})}^{1-a}.
\end{multline}
Observe that $\mathrm{supp}f\cap\mathscr{A}_k\subset\cup_{j=-n_k}^{n_k-1}\mathscr{A}_{k,j}$, where $n_k=2^{2(n_0+1-k)}$ and $|\mathscr{A}_{k,j}|=C2^{kQ}$ for all $k$ and $j$. Using the fact that, $d(\xi)\geq |z|>2^k$ in $\mathscr{A}_{k,j}$ for all $j$, we obtain
\begin{multline}\label{eq 3.11}
\int_{\mathscr{A}_{k,j}}d(\xi)^{\tau\gamma}|z|^{\tau\gamma(p-2)\frac{\alpha}{ps}}|f|^{\tau}d\xi \leq C2^{k\left(Q+\tau\gamma\left(1+(p-2)\frac{\alpha}{ps}\right)\right)}\left(\fint_{\mathscr{A}_{k,j}}|f-(f)_{\mathscr{B}_{k,j}}|^{\tau}+|(f)_{\mathscr{B}_{k,j}}| ^{\tau}\right)\\
\leq C[f]_{s,p,\alpha,\mathscr{B}_{k,j}}^{a\tau}\lVert f\rVert_{L^q(\mathscr{B}_{k,j})}^{(1-a)\tau}+C\,2^{k\left(Q+\tau\gamma\left(1+(p-2)\frac{\alpha}{ps}\right)\right)}|(f)_{\mathscr{B}_{k,j}}|^{\tau}.
\end{multline}
For $t_1,t_2\geq0$ with $t_1+t_2\geq1$, and for $x_k\geq0$ and $y_k\geq0$, one has
\begin{equation}\label{3.24}
\sum_{k\in\ZZ} x_k^{t_1}y_k^{t_2}\leq \left(\sum_{k\in\ZZ}x_k\right)^{t_1}\left(\sum_{k\in\ZZ}y_k\right)^{t_2}.
\end{equation}
We have $\frac{1}{\tau}\leq \frac{a}{p}+\frac{1-a}{q}$, which follows from \eqref{tau gamma rel} and \eqref{range of gamma}. Summation with respect to $j$ from $-n_k$ to $n_k-1$ on \eqref{eq 3.11} and then using the inequality \eqref{3.24} with $t_1=\frac{\tau a}{p}$ and $t_2=\frac{\tau(1-a)}{q}$ we get
\begin{multline}\label{eq 3.13}
\int_{\mathscr{A}_{k}}d(\xi)^{\tau\gamma}|z|^{\tau\gamma(p-2)\frac{\alpha}{ps}}|f|^{\tau}d\xi\leq C\left(\sum_{j=-n_k}^{n_k-1}[f]_{s,p,\alpha,\mathscr{B}_{k,j}}^p\right)^{\frac{a\tau}{p}}\left(\sum_{j=-n_k}^{n_k-1}\lVert f\rVert_{L^q(\mathscr{B}_{k,j})}^q\right)^{\frac{(1-a)\tau}{q}}\\
+C2^{k\left(Q+\tau\gamma\left(1+(p-2)\frac{\alpha}{ps}\right)\right)}\sum_{j=-n_k}^{n_k-1}|(f)_{\mathscr{B}_{k,j}}|^{\tau}.
\end{multline}
Let $m\in\ZZ$ be such that $n_0-m\geq2$. Again summation with respect to $k$ from $m$ to $n_0$ and then \eqref{3.24} gives
\begin{multline*}
\int_{\{\xi\in\HH^n:|z|>2^m\}} d(\xi)^{\tau\gamma}|z|^{\tau\gamma(p-2)\frac{\alpha}{ps}}|f|^{\tau}d\xi\\
\leq C\left(\sum_{k=m}^{n_0}\sum_{j=-n_k}^{n_k-1}[f]_{s,p,\alpha,\mathscr{B}_{k,j}}^p\right)^{\frac{a\tau}{p}}\left(\sum_{k=m}^{n_0}\sum_{j=-n_k}^{n_k-1}\lVert f\rVert_{L^q(\mathscr{B}_{k,j})}^q\right)^{\frac{(1-a)\tau}{q}}\\
+C\sum_{k=m}^{n_0}2^{k\left(Q+\tau\gamma\left(1+(p-2)\frac{\alpha}{ps}\right)\right)}\sum_{j=-n_k}^{n_k-1}|(f)_{\mathscr{B}_{k,j}}|^{\tau}.
\end{multline*}
It follows from Remark \ref{remark 1},
\begin{multline}\label{eq 3.18}
\int_{\{\xi\in\HH^n:|z|>2^m\}} d(\xi)^{\tau\gamma}|z|^{\tau\gamma(p-2)\frac{\alpha}{ps}}|f|^{\tau}d\xi\\
\leq C[f]_{s,p,\alpha}^{a\tau}\lVert f\rVert_{L^q(\HH^n)}^{(1-a)\tau}+C\sum_{k=m}^{n_0}2^{k\left(Q+\tau\gamma\left(1+(p-2)\frac{\alpha}{ps}\right)\right)}\sum_{j=-n_k}^{n_k-1}|(f)_{\mathscr{B}_{k,j}}|^{\tau}.
\end{multline}
Observe that, $n_k=4n_{k+1}$. For $-n_{k+1}\leq j\leq n_{k+1}-1$ and for $i=0,1,2$ and $3$, we have
\begin{multline*}
(f)_{\mathscr{B}_{k,4j+i}}-(f)_{\mathscr{B}_{k+1,j}}=\frac{1}{|\mathscr{B}_{k,4j+i}|}\int_{\mathscr{B}_{k,4j+i}}f-\frac{1}{|\mathscr{B}_{k+1,j}|}\int_{\mathscr{B}_{k+1,j}}f\\
=\frac{1}{|\mathscr{B}_{k,4j+i}|}\int_{\mathscr{B}_{k,4j+i}}\left(f-(f)_{\mathscr{C}_{k+1,j}\cup\mathscr{C}_{k,4j+i}}\right)-\frac{1}{|\mathscr{B}_{k+1,j}|}\int_{\mathscr{B}_{k+1,j}}\left(f-(f)_{\mathscr{C}_{k+1,j}\cup\mathscr{C}_{k,4j+i}}\right).
\end{multline*}
The set $\mathscr{C}_{k,j}$ is defined in Remark \ref{remark 4}. Since $|\mathscr{B}_{k+1,j}|\sim|\mathscr{B}_{k,4j+i}|$, we have
\begin{multline}
|(f)_{\mathscr{B}_{k,4j+i}}-(f)_{\mathscr{B}_{k+1,j}}|\leq \frac{C}{|\mathscr{B}_{k+1,j}\cup\mathscr{B}_{k,4j+i}|}\int_{\mathscr{B}_{k+1,j}\cup\mathscr{B}_{k,4j+i}}\left|f-(f)_{\mathscr{C}_{k+1,j}\cup\mathscr{C}_{k,4j+i}}\right|\\
\leq C\left(\fint_{\mathscr{B}_{k+1,j}\cup\mathscr{B}_{k,4j+i}}\left|f-(f)_{\mathscr{C}_{k+1,j}\cup\mathscr{C}_{k,4j+i}}\right|^\tau\right)^{\frac{1}{\tau}},
\end{multline}
using Jensen's inequality. Inequality \eqref{3.8} gives 
\begin{equation}\label{eq 3.19}
|(f)_{\mathscr{B}_{k,4j+i}}-(f)_{\mathscr{B}_{k+1,j}}|\leq C\,2^{-\frac{k}{\tau}\left(Q+\tau\gamma\left(1+(p-2)\frac{\alpha}{ps}\right)\right)}[f]_{s,p,\alpha,\mathscr{C}_{k+1,j}\cup\mathscr{C}_{k,4j+i}}^{a}\lVert f\rVert_{L^q(\mathscr{C}_{k+1,j}\cup\mathscr{C}_{k,4j+i})}^{(1-a)}.
\end{equation}
For simplicity of the notation we denote $Q'=Q+\tau\gamma\left(1+(p-2)\frac{\alpha}{ps}\right)$. Note that $Q'>2$, which follows from the assumption $1/\tau+\gamma\left(1+(p-2)\frac{\alpha}{ps}\right)/{Q}>\frac{2}{\tau Q}$. Let $\delta=2/\big(1+2^{Q'-2}\big)<1$. Applying Lemma \ref{lemma 4.1} with $c=\delta 2^{Q'-2}>1$ we obtain
$$|(f)_{\mathscr{B}_{k,4j+i}}|^\tau\leq\delta 2^{Q'-2}|(f)_{\mathscr{B}_{k+1,j}}|^\tau\\
+ C\,2^{-kQ'}[f]_{s,p,\alpha,\mathscr{C}_{k+1,j}\cup\mathscr{C}_{k,4j+i}}^{a\tau}\lVert f\rVert_{L^q(\mathscr{C}_{k+1,j}\cup\mathscr{C}_{k,4j+i})}^{(1-a)\tau}.$$
Therefore
\begin{equation}
2^{kQ'}\sum_{i=0}^3|(f)_{\mathscr{B}_{k,4j+i}}|^\tau\leq \delta2^{(k+1)Q'}|(f)_{\mathscr{B}_{k+1,j}}|^\tau+C\sum_{i=0}^3 [f]_{s,p,\alpha,\mathscr{C}_{k+1,j}\cup\mathscr{C}_{k,4j+i}}^{a\tau}\lVert f\rVert_{L^q(\mathscr{C}_{k+1,j}\cup\mathscr{C}_{k,4j+i})}^{(1-a)\tau}.
\end{equation}
Note that
\begin{equation}\label{sum}
\sum_{j=-n_{k+1}}^{n_{k+1}-1}\sum_{i=0}^3a_{k,4j+i}=\sum_{j=-n_k}^{n_k-1}a_{k,j}.
\end{equation}
Summation with respect to $j$ from $-n_{k+1}$ to $n_{k+1}-1$ gives
\begin{multline}\label{eq 3.25}
2^{kQ'}\sum_{j=-n_k}^{n_k-1}|(f)_{\mathscr{B}_{k,j}}|^\tau\leq \delta2^{(k+1)Q'}\sum_{j=-n_{k+1}}^{n_{k+1}-1}|(f)_{\mathscr{B}_{k+1,j}}|^\tau\\
+C\sum_{j=-n_{k+1}}^{n_{k+1}-1} \sum_{i=0}^3 [f]_{s,p,\alpha,\mathscr{C}_{k+1,j}\cup\mathscr{C}_{k,4j+i}}^{a\tau}\lVert f\rVert_{L^q(\mathscr{C}_{k+1,j}\cup\mathscr{C}_{k,4j+i})}^{(1-a)\tau}.
\end{multline}
Again summing \eqref{eq 3.25} with respect to $k$, we obtain
\begin{multline*}
\sum_{k=m}^{n_0}2^{kQ'}\sum_{j=-n_k}^{n_k-1}|(f)_{\mathscr{B}_{k,j}}|^\tau-\delta\sum_{k=m}^{n_0}2^{(k+1)Q'}\sum_{j=-n_{k+1}}^{n_{k+1}-1}|(f)_{\mathscr{B}_{k+1,j}}|^\tau\\
\leq C\sum_{k=m}^{n_0}\sum_{j=-n_{k+1}}^{n_{k+1}-1} \sum_{i=0}^3 [f]_{s,p,\alpha,\mathscr{C}_{k+1,j}\cup\mathscr{C}_{k,4j+i}}^{a\tau}\lVert f\rVert_{L^q(\mathscr{C}_{k+1,j}\cup\mathscr{C}_{k,4j+i})}^{(1-a)\tau},
\end{multline*}
which gives
\begin{multline*}
2^{mQ'}\sum_{j=-n_m}^{n_m-1}|(f)_{\mathscr{B}_{m,j}}|^\tau+\sum_{k=m+1}^{n_0}(1-{\delta})2^{kQ'}\sum_{j=-n_k}^{n_k-1}|(f)_{\mathscr{B}_{k,j}}|^\tau\\
\leq C\sum_{k=m}^{n_0}\sum_{j=-n_{k+1}}^{n_{k+1}-1} \sum_{i=0}^3 [f]_{s,p,\alpha,\mathscr{C}_{k+1,j}\cup\mathscr{C}_{k,4j+i}}^{a\tau}\lVert f\rVert_{L^q(\mathscr{C}_{k+1,j}\cup\mathscr{C}_{k,4j+i})}^{(1-a)\tau},
\end{multline*}
and it follows that
\begin{equation}\label{eq 3.20}
\sum_{k=m}^{n_0}(1-{\delta})2^{kQ'}\sum_{j=-n_k}^{n_k-1}|(f)_{\mathscr{B}_{k,j}}|^\tau\leq C\sum_{k=m}^{n_0}\sum_{j=-n_{k+1}}^{n_{k+1}-1} \sum_{i=0}^3 [f]_{s,p,\alpha,\mathscr{C}_{k+1,j}\cup\mathscr{C}_{k,4j+i}}^{a\tau}\lVert f\rVert_{L^q(\mathscr{C}_{k+1,j}\cup\mathscr{C}_{k,4j+i})}^{(1-a)\tau}.
\end{equation}
We apply \eqref{3.24} thrice on R.H.S and by a similar argument as discussed in Remark \ref{remark 1}, there is a constant $C$, independent of $k$ and $j$ such that
\begin{multline}\label{3.20 a}
\sum_{k=m}^{n_0}\sum_{j=-n_{k+1}}^{n_{k+1}-1} \sum_{i=0}^3 [f]_{s,p,\alpha,\mathscr{C}_{k+1,j}\cup\mathscr{C}_{k,4j+i}}^{a\tau}\lVert f\rVert_{L^q(\mathscr{C}_{k+1,j}\cup\mathscr{C}_{k,4j+i})}^{(1-a)\tau}\\
\leq \left(\sum_{k=m}^{n_0}\sum_{j=-n_{k+1}}^{n_{k+1}-1} \sum_{i=0}^3 [f]_{s,p,\alpha,\mathscr{C}_{k+1,j}\cup\mathscr{C}_{k,4j+i}}^p\right)^{\frac{a\tau}{p}}\left(\sum_{k=m}^{n_0}\sum_{j=-n_{k+1}}^{n_{k+1}-1} \sum_{i=0}^3\lVert f\rVert_{L^q(\mathscr{C}_{k+1,j}\cup\mathscr{C}_{k,4j+i})}^q\right)^{\frac{(1-a)\tau}{q}}\\
\leq C[f]_{s,p,\alpha}^{a\tau}\lVert f\rVert_{L^q(\HH^n)}^{(1-a)\tau}.
\end{multline}
Combining \eqref{eq 3.18}, \eqref{eq 3.20} and \eqref{3.20 a}, we get
\begin{equation}\label{eq 3.21}
\int_{\{\xi\in\HH^n:|z|>2^m\}} d(\xi)^{\tau\gamma}|z|^{\tau\gamma(p-2)\frac{\alpha}{ps}}|f|^{\tau}d\xi\leq C[f]_{s,p,\alpha}^{a\tau}\lVert f\rVert_{L^q(\HH^n)}^{(1-a)\tau}.
\end{equation}
Inequality \eqref{eq frac CKN in H^n i} follows by passing the limit $m\rightarrow-\infty$.\smallskip

\textbf{Step 2:} Proof of (ii). Choose $m\in\ZZ$ such that 
$$\mathrm{supp}(f)\cap \{\xi=(z,t):|z|<2^m\}=\phi.$$
Let $\delta=\big(1+2^{Q'-2}\big)/2$. Then $\delta<1$. Applying Lemma \ref{lemma 4.1} in \eqref{eq 3.19} with $c=\delta/2^{Q'-2}>1$, we get
$$|(f)_{\mathscr{B}_{k+1,j}}|^\tau\leq\delta 2^{-\left(Q'-2\right)}|(f)_{\mathscr{B}_{k,4j+i}}|^\tau+C\,2^{-kQ'}[f]_{s,p,\alpha,\mathscr{C}_{k+1,j}\cup\mathscr{C}_{k,4j+i}}^{a\tau}\lVert f\rVert_{L^q(\mathscr{C}_{k+1,j}\cup\mathscr{C}_{k,4j+i})}^{(1-a)\tau}.$$
Summation with respect to $i$ gives
$$2^{(k+1)Q'}|(f)_{\mathscr{B}_{k+1,j}}|^\tau\leq\delta 2^{kQ'}\sum_{i=0}^3|(f)_{\mathscr{B}_{k,4j+i}}|^\tau+C\sum_{i=0}^3[f]_{s,p,\alpha,\mathscr{C}_{k+1,j}\cup\mathscr{C}_{k,4j+i}}^{a\tau}\lVert f\rVert_{L^q(\mathscr{C}_{k+1,j}\cup\mathscr{C}_{k,4j+i})}^{(1-a)\tau}.$$
Again summing with respect to $j$ from $-n_{k+1}$ to $n_{k+1}-1$ and using \eqref{sum} we obtain
\begin{multline}\label{eq 3.23}
2^{(k+1)Q'}\sum_{j=-n_{k+1}}^{n_{k+1}-1}|(f)_{\mathscr{B}_{k+1,j}}|^\tau\leq \delta2^{kQ'}\sum_{j=-n_k}^{n_k-1}|(f)_{\mathscr{B}_{k,j}}|^\tau\\
+C\sum_{j=-n_{k+1}}^{n_{k+1}-1} \sum_{i=0}^3 [f]_{s,p,\alpha,\mathscr{C}_{k+1,j}\cup\mathscr{C}_{k,4j+i}}^{a\tau}\lVert f\rVert_{L^q(\mathscr{C}_{k+1,j}\cup\mathscr{C}_{k,4j+i})}^{(1-a)\tau}.
\end{multline}
Finally summation over $k$ from $m-1$ to $n_0-1$ and 
\eqref{3.20 a} gives
\begin{equation}\label{eq 3.24}
\sum_{k=m}^{n_0}(1-\delta)2^{kQ'}\sum_{j=-n_k}^{n_k-1}|(f)_{\mathscr{B}_{k,j}}|^\tau\leq C[f]_{s,p,\alpha}^{a\tau}\lVert f\rVert_{L^q(\HH^n)}^{(1-a)\tau}.
\end{equation}
Inequality \eqref{eq frac CKN in H^n ii} can be obtained by applying \eqref{eq 3.24} in \eqref{eq 3.18}. This completes the proof.
\end{proof}
\section{Proof of Theorem \ref{frac CKN in H^n b}}
For the proof of Theorem \ref{frac CKN in H^n b} we need the following two lemmas established in \cite{Adi-Mallick}.
%(see Lemma 4.1 and Lemma 4.2).
\begin{lemma}\label{lemma 3.3}
Let $1\leq p<\infty$, $s\in(0,1)$ and $\alpha\in\RR$ be such that $(p-2)\alpha\geq 0$ and let $E\subset \HH^n$ be any measurable set with finite Lebesgue measure, i.e. $|E|<\infty$. Fix $\xi=(z,t)\in\HH^n$. Then there exists a positive constant $C$, depends only on $n$, $p$, $s$ and $\alpha$ such that
\begin{equation}\label{3.10}
\int_{E^C}\frac{d\xi'}{d({\xi}^{-1}\circ \xi')^{Q+ps} |z-z'|^{(p-2)\alpha}}\geq C|E|^{-\frac{ps+(p-2)\alpha}{Q}},
\end{equation}
where $\xi'=(z',t')\in E^C$.
\end{lemma}
\begin{lemma}\label{lemma 3.4}
Let $1\leq p<\infty$, $s\in(0,1)$ and $\alpha\in\RR$ be such that $(p-2)\alpha\geq 0$ and $ps+(p-2)\alpha< Q$. Let $\{e_i\}$ and $\{d_i\}$ be two non-negative sequences of real numbers satisfying the following properties:\\
\emph{(a)} $e_i$ is decreasing;\\
\emph{(b)} $e_i=\sum_{j=i}^{\infty}d_j$.\\
Then the followings hold.\\
\emph{(i)} For any fixed $T>0$ we have
\begin{equation}\label{3.11}
\sum_{i\in\ZZ}e_i^{\frac{Q-ps-(p-2)\alpha}{Q}}T^i\leq T^{\frac{Q}{Q-ps-(p-2)\alpha}}\sum_{i\in\ZZ,e_i\neq0}e_{i+1}e_i^{-\frac{ps+(p-2)\alpha}{Q}}T^i.
\end{equation}
\emph{(ii)} Moreover, if $T>1$ then
\begin{equation}\label{3.12}
\sum_{i\in\ZZ,e_{i-1}\neq0}\, \sum_{j\geq i+1}T^{pi}e_{i-1}^{-\frac{ps+(p-2)\alpha}{Q}}d_j\leq \frac{1}{T^p-1}\sum_{i\in\ZZ,e_{i-1}\neq0}T^{pi}e_{i-1}^{-\frac{ps+(p-2)\alpha}{Q}}d_i.
\end{equation}
\end{lemma}
We next need the following  lemma to prove our theorem.
\begin{lemma}\label{lemma 3.5}
Let $\beta\leq0$ be such that $Q-2+\beta>0$. Then for any measurable set $D\subset B_{\lambda}(0)$, we have
\begin{equation}\label{3.13}
\int_D|z|^\beta dzdt\leq C\lambda^2|D|^{\frac{Q-2+\beta}{Q}},
\end{equation}
where the constant $C$ is independent of $\lambda$.
\end{lemma}
\begin{proof}
Let $R>0$ be such that $|D|=2\omega_{2n}R^Q=|D_R(0)|$, where the set $D_R(0)$ is defined in \eqref{set D_r}. Note that $R\leq C\lambda$. Now
$$I:=\int_D|z|^\beta dzdt=\int_{D\cap D_R(0)}|z|^\beta dzdt+\int_{D\cap D_R(0)^C}|z|^\beta dzdt.$$
Let us define the sets
$$D_{R,1}:=\{z\in\RR^{2n}:|z|>R\}\times\{t\in\RR:|t|>R^2\},$$
$$D_{R,2}:=\{z\in\RR^{2n}:|z|>R\}\times\{t\in\RR:|t|<R^2\},$$
$$D_{R,3}:=\{z\in\RR^{2n}:|z|<R\}\times\{t\in\RR:|t|>R^2\},$$
and the quantities
$$I_0:=\int_{D\cap D_R(0)}|z|^\beta dzdt,$$
$$I_i:=\int_{D\cap D_{R,i}}|z|^\beta dzdt,\ \ i=1,2,3.$$
Clearly, $D_R(0)^C=D_{R,1}\cup D_{R,2}\cup D_{R,3}$ and $I=I_0+I_1+I_2+I_3$. We will estimate the integrals separately.\\
\textbf{Estimate for $I_0$:}
\begin{multline}
I_0\leq\int_{D_R(0)}|z|^\beta dzdt= C\int_0^R\int_0^{R^2}r^{Q-3+\beta}drdt\\
=CR^2\frac{R^
{Q-2+\beta}}{Q-2+\beta}\leq C\lambda^2R^
{Q-2+\beta}=C\lambda^2|D|^{\frac{Q-2+\beta}{Q}}.
\end{multline}
\textbf{Estimate for $I_1$:}
\begin{equation}
I_1=\int_{D\cap D_{R,1}}|z|^\beta dzdt\leq R^\beta|D|=CR^{Q+\beta}=\leq C\lambda^2|D|^{\frac{Q-2+\beta}{Q}}.
\end{equation}
\textbf{Estimate for $I_2$:}
\begin{equation}
I_2=\int_{D\cap D_{R,2}}|z|^\beta dzdt\leq R^\beta|D|=CR^{Q+\beta}=\leq C\lambda^2|D|^{\frac{Q-2+\beta}{Q}}.
\end{equation}
\textbf{Estimate for $I_3$:}
\begin{equation}
I_3\leq C\int_0^R\int_0^{\lambda^2}r^{Q-3+\beta}drdt\leq C\lambda^2\frac{R^
{Q-2+\beta}}{Q-2+\beta}=C\lambda^2|D|^{\frac{Q-2+\beta}{Q}}.
\end{equation}
This completes the proof.
\end{proof}
Our next proposition  is a fractional CKN type inequality in a bounded domain which doesn't contain a ball around origin.
\begin{proposition}\label{prop 4.1}
Let $p$, $s$, $\alpha$, $\tau$, $a$, $\gamma$ and $q$ be as in Theorem \emph{\ref{frac CKN in H^n}} satisfying the relation \eqref{tau gamma rel}, \eqref{range of gamma} and $Q+\tau\gamma\left(1+(p-2)\frac{\alpha}{ps}\right)\geq0$. In addition we assume $\tau\gamma\leq-2$ when $Q+\tau\gamma\left(1+(p-2)\frac{\alpha}{ps}\right)\\>0$ and $\tau\gamma<-2$ when $Q+\tau\gamma\left(1+(p-2)\frac{\alpha}{ps}\right)=0$. Let $\Omega$ be a bounded open set in $\HH^n$ such that $\Omega\cap B_r(0)=\phi$ for some $r>0$. Then there exists a positive constant $C$ which depends on $n$, $p$, $s$, $\alpha$, $r$ and $\Omega$ such that for $q>\frac{(1-a)\tau Q}{Q-2+\tau\gamma(p-2)\frac{\alpha}{ps}}$, we have
\begin{equation}\label{3.18a}
\left(\int_{\Omega}d(\xi)^{\tau\gamma}|z|^{\tau\gamma(p-2)\frac{\alpha}{ps}}|f(\xi)|^{\tau}d\xi\right)^{\frac{1}{\tau}}\leq C[f]_{s,p,\alpha,\Omega}^a\lVert f\rVert_{L^q(\Omega)}^{1-a}\ \forall f\in C_c(\Omega).
\end{equation}
\end{proposition}
\begin{proof}
It is enough to prove \eqref{3.18a} for $0\leq f\in C_c(\Omega)$. For $i \in \ZZ,$ define the following sets:
$$E_i:=\{\xi\in\Omega: f(\xi)\geq2^i\},$$
$$D_i:=E_i\setminus E_{i+1}=\{\xi\in\Omega: 2^i\leq f(\xi)< 2^{i+1}\}.$$
Let $e_i=|E_i|$ and $d_i=|D_i|$. Clearly $d_i\leq e_i$ and they satisfy the conditions of Lemma \ref{lemma 3.4}.\smallskip

Relation \eqref{tau gamma rel} gives
\begin{equation}\label{4.10}
\frac{Q-2+\tau\gamma(p-2)\frac{\alpha}{ps}}{\tau Q}=a\left(\frac{Q-ps'-(p-2)\alpha}{pQ}\right)+\frac{1-a}{q},
\end{equation}
where $0<s'=s+\frac{\gamma}{a}+\frac{2}{a\tau}\leq s$ and $Q-ps'-(p-2)\alpha>0$ follows from the range of $q$. We first estimate the quantity $[f]_{s',p,\alpha,\Omega}$. Since $f\in C_c(\HH^n)$, for any $i,j\in \ZZ$ with $j\leq i-2$, if $\xi=(z,t)\in D_i$ and $\xi'=(z',t')\in D_j$, then we have $2^j\leq f(\xi')\leq 2^{j+1}\leq 2^{i-1}<2^i\leq f(\xi)\leq 2^{i+1}$, which gives $f(\xi)-f(\xi')\geq 2^{i-1}.$ Hence
\begin{multline*}
[f]_{s',p,\alpha,\Omega}^p\geq\sum_{i\in\ZZ}\, \sum_{j\leq i-2}\int_{D_i}\int_{D_j}\frac{|f(\xi)-f(\xi')|^pd\xi' d\xi}{d({\xi}^{-1}\circ \xi')^{Q+ps'} |z-z'|^{(p-2)\alpha}}\\
\geq\sum_{i\in\ZZ}\, \sum_{j\leq i-2}2^{p(i-1)}\int_{D_i}\int_{D_j}\frac{d\xi' d\xi}{d({\xi}^{-1}\circ \xi')^{Q+ps'} |z-z'|^{(p-2)\alpha}}\\
=\sum_{i\in\ZZ}2^{p(i-1)}\int_{D_i}\int_{\cup_{j\leq i-2}\,D_j}\frac{d\xi' d\xi}{d({\xi}^{-1}\circ \xi')^{Q+ps'} |z-z'|^{(p-2)\alpha}}\\
=\sum_{i\in\ZZ}2^{p(i-1)}\int_{D_i}\int_{E_{i-1}^C}\frac{d\xi' d\xi}{d({\xi}^{-1}\circ \xi')^{Q+ps'} |z-z'|^{(p-2)\alpha}}.
\end{multline*}
Using Lemma \ref{lemma 3.3} we obtain
\begin{equation}\label{3.19}
[f]_{s',p,\alpha,\Omega}^p\geq C\sum_{i\in\ZZ,e_{i-1}\neq0}2^{p(i-1)}e_{i-1}^{-\frac{ps'+(p-2)\alpha}{Q}}d_i.
\end{equation}
Putting $d_i=e_i-\sum_{j\geq i+1}d_j$ and then using Lemma \ref{lemma 3.4} (ii) with $T=2$, we get
\begin{multline*}
[f]_{s',p,\alpha,\Omega}^p\geq \frac{C}{2^p}\left(\sum_{i\in\ZZ,e_{i-1}\neq0}2^{pi}e_ie_{i-1}^{-\frac{ps'+(p-2)\alpha}{Q}}-\sum_{i\in\ZZ,e_{i-1}\neq0}\, \sum_{j\geq i+1}2^{pi}e_{i-1}^{-\frac{ps'+(p-2)\alpha}{Q}}d_j\right)\\
\geq \frac{C}{2^p}\left(\sum_{i\in\ZZ,e_{i-1}\neq0}2^{pi}e_ie_{i-1}^{-\frac{ps'+(p-2)\alpha}{Q}}-\frac{1}{2^p-1}\sum_{i\in\ZZ,e_{i-1}\neq0}2^{pi}e_{i-1}^{-\frac{ps'+(p-2)\alpha}{Q}}d_i\right)\\
\geq \frac{C}{2^p}\sum_{i\in\ZZ,e_{i-1}\neq0}2^{pi}e_ie_{i-1}^{-\frac{ps'+(p-2)\alpha}{Q}}-\frac{1}{2^p-1}[f]_{s',p,\alpha,\Omega}^p,\ \left[\mathrm{using}\ \mathrm{inequality}\ \eqref{3.19}\right]
\end{multline*}
which gives
$$[f]_{s',p,\alpha,\Omega}^p\geq\frac{2^p-1}{2^{2p}}C\sum_{i\in\ZZ,e_{i-1}\neq0}2^{pi}e_ie_{i-1}^{-\frac{ps'+(p-2)\alpha}{Q}}.$$
By Lemma \ref{lemma 3.4} (i) with $T=2^p$ we get
\begin{equation}\label{3.20}
[f]_{s',p,\alpha,\Omega}^p\geq C\sum_{i\in\ZZ}2^{pi}e_i^{\frac{Q-ps'-(p-2)\alpha}{Q}}.
\end{equation}
On the other hand,
\begin{equation}\label{3.21}
\int_{\Omega}|f(\xi)|^q d\xi=\sum_{i\in\ZZ}\int_{D_i}|f(\xi)|^q d\xi\geq\sum_{i\in\ZZ}2^{qi}d_i.
\end{equation}
Now let
$$J:=\int_{\Omega}|z|^{\tau\gamma(p-2)\frac{\alpha}{ps}}|f(\xi)|^{\tau}d\xi=\sum_{i\in\ZZ}\int_{D_i}|z|^{\tau\gamma(p-2)\frac{\alpha}{ps}}|f(\xi)|^{\tau}d\xi\leq\sum_{i\in\ZZ}2^{\tau(i+1)}\int_{D_i}|z|^{\tau\gamma(p-2)\frac{\alpha}{ps}}d\xi.$$
It follows from the assumptions that $Q-2+\tau\gamma(p-2)\frac{\alpha}{ps}>0$. Therefore, from Lemma \ref{lemma 3.5} with $\beta=\tau\gamma(p-2)\frac{\alpha}{ps},$ we get
\begin{equation}\label{3.22}
J\leq C\sum_{i\in\ZZ}2^{\tau i}d_i^{\frac{Q-2+\tau\gamma(p-2)\frac{\alpha}{ps}}{Q}}.
\end{equation}
Relation \eqref{4.10} gives
\begin{equation}\label{3.23}
J\leq C\sum_{i\in\ZZ}\left(2^{pi}d_i^{\frac{Q-ps'-(p-2)\alpha}{Q}}\right)^{\frac{a\tau}{p}}\left(2^{qi}d_i\right)^{\frac{(1-a)\tau}{q}}.
\end{equation}
By inequality \eqref{3.24} and by the fact $d_i\leq e_i$ we obtain
\begin{equation}\label{3.25}
J\leq C\left(\sum_{i\in\ZZ}2^{pi}e_i^{\frac{Q-ps'-(p-2)\alpha}{Q}}\right)^{\frac{a\tau}{p}}\left(\sum_{i\in\ZZ}2^{qi}d_i\right)^{\frac{(1-a)\tau}{q}}.
\end{equation}
Inequality \eqref{3.25}, \eqref{3.20} and \eqref{3.21} together imply
\begin{equation}\label{4.17}
\left(\int_{\Omega}|z|^{\tau\gamma(p-2)\frac{\alpha}{ps}}|f(\xi)|^{\tau}d\xi\right)^{\frac{1}{\tau}}\leq C [f]_{s',p,\alpha,\Omega}^a\lVert f\rVert_{L^q(\Omega)}^{1-a}.
\end{equation}
Inequality \eqref{3.18a} follows from \eqref{4.17} and the facts that $d(\xi)\geq r$ in $\Omega$ and $[f]_{s',p,\alpha,\Omega}\leq C[f]_{s,p,\alpha,\Omega}$.
\end{proof}
\begin{corollary}\label{coro 2}
Let $p$, $s$, $\alpha$, $\tau$, $a$, $\gamma$ and $q$ be as in Theorem \emph{\ref{frac CKN in H^n}}  satisfying the relation \eqref{tau gamma rel}, \eqref{range of gamma} and $Q+\tau\gamma\left(1+(p-2)\frac{\alpha}{ps}\right)\geq0$. In addition we assume $\tau\gamma\leq-2$ when $Q+\tau\gamma\left(1+(p-2)\frac{\alpha}{ps}\right)>0$ and $\tau\gamma<-2$ when $Q+\tau\gamma\left(1+(p-2)\frac{\alpha}{ps}\right)=0$. For $0<r<R$ consider the set $A_{r,R}=A_{r,R}(0)$ as defined in \eqref{set A_r,R}. Then there exists a constant $C>0$ depends on $n$, $p$, $s$, $\alpha$, $r$ and $R$ such that for $q>\frac{(1-a)\tau Q}{Q-2+\tau\gamma(p-2)\frac{\alpha}{ps}}$, we have
\begin{equation}\label{3.18}
\left(\int_{A_{r,R}}d(\xi)^{\tau\gamma}|z|^{\tau\gamma(p-2)\frac{\alpha}{ps}}|f(\xi)|^{\tau}d\xi\right)^{\frac{1}{\tau}}\leq C \lVert f\rVert_{s,p,\alpha,A_{\frac{r}{2},2R}}^a\lVert f\rVert_{L^q(A_{\frac{r}{2},2R})}^{1-a},\ \forall f\in C(\overline{{A}_{\frac{r}{2},2R}}).
\end{equation}
\end{corollary}
\begin{proof}
Let $\phi\in C_c^\infty(\HH^n)$ be such that $0\leq\phi\leq1$, $\phi\equiv1$ in $A_{r,R}$ and $\phi\equiv0$ in $A_{\frac{3r}{4},\frac{3R}{2}}^C$. Clearly $f\phi\in C_c(A_{\frac{r}{2},2R})$. By Proposition \ref{prop 4.1},
\begin{multline}\label{4.19}
\left(\int_{A_{r,R}}d(\xi)^{\tau\gamma}|z|^{\tau\gamma(p-2)\frac{\alpha}{ps}}|f(\xi)|^{\tau}d\xi\right)^{\frac{1}{\tau}}\\
\leq \left(\int_{A_{\frac{r}{2},2R}}d(\xi)^{\tau\gamma}|z|^{\tau\gamma(p-2)\frac{\alpha}{ps}}|f(\xi)|^{\tau}|\phi(\xi)|^{\tau} d\xi\right)^{\frac{1}{\tau}}
\leq C [f\phi]_{s,p,\alpha,A_{\frac{r}{2},2R}}^a\lVert f\phi\rVert_{L^q(A_{\frac{r}{2},2R})}^{1-a}.
\end{multline}
Inequality \eqref{3.18} follows from \eqref{4.19}, \eqref{3.4} and the fact $0\leq\phi\leq1$.
\end{proof}
\begin{lemma}\label{lemma 3.7}
Let $p$, $s$, $\alpha$, $\tau$, $a$, $\gamma$ and $q$ be as in Theorem \emph{\ref{frac CKN in H^n}} satisfying the relation \eqref{tau gamma rel}, \eqref{range of gamma} and $Q+\tau\gamma\left(1+(p-2)\frac{\alpha}{ps}\right)\geq0$. In addition we assume $\tau\gamma\leq-2$ when $Q+\tau\gamma\left(1+(p-2)\frac{\alpha}{ps}\right)>0$ and $\tau\gamma<-2$ when $Q+\tau\gamma\left(1+(p-2)\frac{\alpha}{ps}\right)=0$. Let $\lambda>0$, $0<r<R$ and denote the sets $A_{\lambda r,\lambda R}(0)$ and $A_{ \frac{\lambda r}{2},2\lambda R}(0)$ by $A_\lambda$ and $F_\lambda$ respectively. Then for $q>\frac{(1-a)\tau Q}{Q-2+\tau\gamma(p-2)\frac{\alpha}{ps}}$ and for $f\in C(\Bar{F}_\lambda)$, we have
\begin{multline}\label{3.26}
\left(\fint_{A_\lambda}d(\xi)^{\tau\gamma}|z|^{\tau\gamma(p-2)\frac{\alpha}{ps}}|f-(f)_{F_\lambda}|^{\tau}\right)^{\frac{1}{\tau}}\\
\leq C\lambda^{a\frac{ps+(p-2)\alpha-Q}{p}+\gamma\left(1+(p-2)\frac{\alpha}{ps}\right)}[f]_{s,p,\alpha,F_\lambda}^a\left(\fint_{F_\lambda}|f|^{q}d\xi\right)^{\frac{1-a}{q}},
\end{multline}
where the constant $C$ is independent of $\lambda$ and $f$.
\end{lemma}
\begin{proof}
For any function $f$ we define the function $\Tilde{f}(\xi)=f(D_\lambda\xi)$. Note that $\Tilde{f}\in C(\overline{A_{\frac{r}{2},2R}})$ if and only if $f\in C(\Bar{F}_\lambda)$. We have by Corollary \ref{coro 2},
\begin{multline*}
\left(\int_{A_{r,R}}d(\xi)^{\tau\gamma}|z|^{\tau\gamma(p-2)\frac{\alpha}{ps}}|\Tilde{f}-(f)_{F_\lambda}|^{\tau}\right)^{\frac{1}{\tau}}\\
\leq C \left(\lVert \Tilde{f}-(f)_{F_\lambda}\rVert_{L^p(A_{\frac{r}{2},2R})}^p+[\Tilde{f}]^p_{s,p,\alpha,A_{\frac{r}{2},2R}}\right)^{\frac{a}{p}}\left(\lVert \Tilde{f}-(f)_{F_\lambda}\rVert_{L^q(A_{\frac{r}{2},2R})}\right)^{1-a}.
\end{multline*}
A change of variable gives
\begin{multline}\label{3.28}
\left(\lambda^{-\tau\gamma\left(1+(p-2)\frac{\alpha}{ps}\right)}\fint_{A_\lambda}d(\xi)^{\tau\gamma}|z|^{\tau\gamma(p-2)\frac{\alpha}{ps}}|f-(f)_{F_\lambda}|^{\tau}\right)^{\frac{1}{\tau}}\\
\leq C\left(\fint_{F_\lambda}|f-(f)_{F_\lambda}|^p+\lambda^{ps+(p-2)\alpha-Q}[f]_{s,p,\alpha,F_\lambda}^p\right)^{\frac{a}{p}}\left(\fint_{F_\lambda}|f|^q\right)^\frac{1-a}{q}.
\end{multline}
One can observe that $d({\xi}^{-1}\circ \xi')\leq C\lambda$ for $\xi,\xi'\in F_\lambda$. A similar argument as for \eqref{eq 3.9} gives
\begin{equation}\label{3.29}
\fint_{F_\lambda}|f-(f)_{F_\lambda}|^p\leq C\lambda^{ps+(p-2)\alpha-Q}[f]_{s,p,\alpha,F_\lambda}^p.
\end{equation}
Inequality \eqref{3.26} follows from \eqref{3.28} and \eqref{3.29}.
\end{proof}
The above lemma is the main ingredient of the proof of Theorem \ref{frac CKN in H^n b}. The technique of proof is similar to that of Theorem \ref{frac CKN in H^n}.
\begin{proof}[\textbf{Proof of Theorem} \ref{frac CKN in H^n b}]
Let $f\in C_c^1(\HH^n)$. We choose $n_0\in \ZZ$ such that $\mathrm{supp}\,f\subset B_{2^{n_0+1}}(0)$. Let us consider the decomposition $\{A_k\}$ of $\HH^n$ defined in \eqref{decom b} and we define the sets $F_k:=A_{2^{k-1},2^{k+2}}(0)$. By Lemma \ref{lemma 3.7} and by the relation \eqref{tau gamma rel} we have
\begin{multline}\label{4.1}
\left(\fint_{A_k}d(\xi)^{\tau\gamma}|z|^{\tau\gamma(p-2)\frac{\alpha}{ps}}|f-(f)_{F_k}|^{\tau} d\xi\right)^{\frac{1}{\tau}}\\
\leq 
C2^{k\left(a\frac{ps+(p-2)\alpha-Q}{p}+\gamma\left(1+(p-2)\frac{\alpha}{ps}\right)\right)}[f]_{s,p,\alpha,F_k}^a\left(\frac{1}{2^{kQ}}\int_{F_k}|f|^{q} d\xi\right)^{\frac{1-a}{q}}\\
=C\,2^{-\frac{kQ}{\tau}}[f]_{s,p,\alpha,F_k}^a\lVert f\rVert_{L^q(F_k)}^{1-a}.
\end{multline}
Applying \eqref{4.1}, Lemma \ref{lemma 3.5} and the fact $|A_k|\sim|F_k|\sim2^{kQ}$, we obtain
\begin{multline}\label{4.2}
\int_{A_k}d(\xi)^{\tau\gamma}|z|^{\tau\gamma(p-2)\frac{\alpha}{ps}}|f|^{\tau}d\xi\\
\leq C2^{kQ}\fint_{A_k}d(\xi)^{\tau\gamma}|z|^{\tau\gamma(p-2)\frac{\alpha}{ps}}|f-(f)_{F_k}|^{\tau}+C2^{k\tau\gamma}|(f)_{F_k}|^{\tau}\int_{A_k}|z|^{\tau\gamma(p-2)\frac{\alpha}{ps}}d\xi\\
\leq C[f]_{s,p,\alpha,F_k}^{a\tau}\lVert f\rVert_{L^q(F_k)}^{(1-a)\tau}+C\,2^{k\left(Q+\tau\gamma\left(1+(p-2)\frac{\alpha}{ps}\right)\right)}|(f)_{F_k}|^{\tau}.
\end{multline}
Summation with respect to $k$ from $m$ to $n_0$, and inequality \eqref{3.24} gives
\begin{multline}\label{4.3}
\int_{d(\xi)>2^m}d(\xi)^{\tau\gamma}|z|^{\tau\gamma(p-2)\frac{\alpha}{ps}}|f|^{\tau}d\xi \\
\leq C\left(\sum_{k=m}^{n_0}[f]_{s,p,\alpha,F_k}^{p}\right)^{\frac{a\tau}{p}}\left(\sum_{k=m}^{n_0}\lVert f\rVert_{L^q(F_k)}^{q}\right)^{\frac{(1-a)\tau}{q}}+C\sum_{k=m}^{n_0}2^{k\left(Q+\tau\gamma\left(1+(p-2)\frac{\alpha}{ps}\right)\right)}|(f)_{F_k}|^{\tau}.
%\leq C[f]_{s,p,\alpha}^{a\tau}\lVert f\rVert_{L^q(\HH^n)}^{(1-a)\tau}+C\sum_{k=m}^{n_0}2^{k\left(Q+\tau\gamma\left(1+(p-2)\frac{\alpha}{ps}\right)\right)}|(f)_{2\mathscr{B}_k}|^{\tau}.
\end{multline}
Observe that $F_k\cup F_{k+1}$ is the set $A_{2^{k-1},2^{k+3}}(0)$. Let us define the set $G_k:=A_{2^{k-2},2^{k+4}}(0)$. Now
$$(f)_{F_k}-(f)_{F_{k+1}}=\frac{1}{|F_k|}\int_{F_k}\left(f-(f)_{G_k}\right)-\frac{1}{|F_{k+1}|}\int_{F_{k+1}}\left(f-(f)_{G_k}\right).$$
Since $|{F_k}|\sim|F_{k+1}|$ and $|z|\leq d(\xi)\leq 2^{k+3}$ in $F_k\cup F_{k+1}$, we have
\begin{multline*}
|(f)_{F_k}-(f)_{F_{k+1}}|\leq \frac{C}{|{F_k}\cup F_{k+1}|}\int_{F_k\cup F_{k+1}}\left|f-(f)_{G_k}\right|\\
\leq C\left(\fint_{F_k\cup F_{k+1}}\left|f-(f)_{G_k}\right|^\tau\right)^{\frac{1}{\tau}}\\
\leq C2^{-k\gamma\left(1+(p-2)\frac{\alpha}{ps}\right)}\left(\fint_{F_k\cup F_{k+1}}d(\xi)^{\tau\gamma}|z|^{\tau\gamma(p-2)\frac{\alpha}{ps}}\left|f-(f)_{G_k}\right|^\tau\right)^{\frac{1}{\tau}}.
\end{multline*}
Since $F_k\cup F_{k+1}\Subset G_k$, a similar argument as \eqref{4.1} implies
\begin{equation}\label{4.4}
|(f)_{F_k}-(f)_{F_{k+1}}|\leq C2^{-\frac{k}{\tau}\left(Q+\tau\gamma\left(1+(p-2)\frac{\alpha}{ps}\right)\right)}[f]_{s,p,\alpha,G_k}^a\lVert f\rVert_{L^q(G_k)}^{(1-a)}.
\end{equation}
Let $\delta=2/(1+2^{Q'})<1$, where $Q'=Q+\tau\gamma\left(1+(p-2)\frac{\alpha}{ps}\right)$. Applying Lemma \ref{lemma 4.1} with $c=\delta2^{Q'}$ we obtain
$$|(f)_{F_k}|^\tau\leq\delta2^{Q'}|(f)_{F_{k+1}}|^\tau+C2^{-kQ'}[f]_{s,p,\alpha,G_k}^{a\tau}\lVert f\rVert_{L^q(G_k)}^{(1-a)\tau},$$
which implies
$$2^{kQ'}|(f)_{F_k}|^\tau\leq\delta2^{(k+1)Q'}|(f)_{F_{k+1}}|^\tau +C[f]_{s,p,\alpha,G_k}^{a\tau}\lVert f\rVert_{L^q(G_k)}^{(1-a)\tau}.$$
Summation with respect to $k$ from $m$ to $n_0$ gives
$$2^{mQ'}|(f)_{F_m}|^\tau+(1-\delta)\sum_{k=m+1}^{n_0}2^{kQ'}|(f)_{F_k}|^\tau\leq C\sum_{k=m}^{n_0}[f]_{s,p,\alpha,G_k}^{a\tau}\lVert f\rVert_{L^q(G_k)}^{(1-a)\tau},$$
and it follows from \eqref{3.24}, 
\begin{equation}\label{4.5}
(1-\delta)\sum_{k=m}^{n_0}2^{kQ'}|(f)_{F_k}|^\tau\leq C\left(\sum_{k=m}^{n_0}[f]_{s,p,\alpha,G_k}^{p}\right)^{\frac{a\tau}{p}}\left(\sum_{k=m}^{n_0}\lVert f\rVert_{L^q(G_k)}^{q}\right)^{\frac{(1-a)\tau}{q}}.
\end{equation}
Inequality \eqref{eq frac CKN in H^n i} follows from \eqref{4.3} and \eqref{4.5} and from the facts
\begin{equation}\label{overlap}
\sum_{k=m}^{n_0}\int_{F_k}f(\xi)d\xi\leq3\int_{\HH^n}f(\xi)d\xi\ \ \mathrm{and}\ \sum_{k=m}^{n_0}\int_{G_k}f(\xi)d\xi\leq5\int_{\HH^n}f(\xi)d\xi
\end{equation}
for any non-negative function $f$.
\end{proof}
\begin{remark}
\emph{If $Q-ps-(p-2)\alpha>0$ and $\tau\gamma<-2$, Theorem \ref{frac CKN in H^n b} can be proved directly by the technique used in \cite{Adi-Mallick} to prove Theorem \ref{frac-har in H^n} as follows:\\
For $0\leq f\in C_c(\Omega)$ and for $i \in \ZZ,$ let us define the sets:
$$E_i:=\{\xi\in\HH^n: f(\xi)\geq2^i\},$$
$$D_i:=E_i\setminus E_{i+1}=\{\xi\in\HH^n: 2^i\leq f(\xi)< 2^{i+1}\},$$
and let $e_i=|E_i|$ and $d_i=|D_i|$. One can show that (see \cite{Adi-Mallick}, Section 5)
$$J:=\int_{\HH^n}d(\xi)^{\tau\gamma}|z|^{\tau\gamma(p-2)\frac{\alpha}{ps}}|f(\xi)|^{\tau}d\xi\leq C\sum_{i\in\ZZ}2^{\tau i}d_i^{\frac{Q+\tau\left(1+\gamma(p-2)\frac{\alpha}{ps}\right)}{Q}}.$$
By the relation \eqref{tau gamma rel} and inequality \eqref{3.24} it follows that
\begin{equation}\label{4.29}
J\leq C\left(\sum_{i\in\ZZ}2^{pi}d_i^{\frac{Q-ps-(p-2)\alpha}{Q}}\right)^{\frac{a\tau}{p}}\left(\sum_{i\in\ZZ}2^{qi}d_i\right)^{\frac{(1-a)\tau}{q}}.
\end{equation}
On the other hand, we have
\begin{equation}\label{4.30}
[f]_{s,p,\alpha}^p\geq C\sum_{i\in\ZZ}2^{pi}e_i^{\frac{Q-ps-(p-2)\alpha}{Q}}\geq C\sum_{i\in\ZZ}2^{pi}d_i^{\frac{Q-ps-(p-2)\alpha}{Q}}.
\end{equation}
and
\begin{equation}\label{4.31}
\int_{\Omega}|f(\xi)|^q d\xi=\sum_{i\in\ZZ}\int_{D_i}|f(\xi)|^q d\xi\geq\sum_{i\in\ZZ}2^{qi}d_i.
\end{equation}
Inequality \eqref{4.29}, \eqref{4.30} and \eqref{4.31} together establish \eqref{eq frac CKN in H^n iii} for the case $Q-ps-(p-2)\alpha>0$ along with $\tau\gamma<-2$.}
\end{remark}

\section{Proof of the limiting case}
Before proving Theorem \ref{theorem limiting case} we discuss the following remark:
\begin{remark}
\emph{In the proof of the limiting case $1/\tau+\gamma\left(1+(p-2)\frac{\alpha}{ps}\right)/{Q}=0$ we will be using Lemma \ref{lemma 3.7}. Though the condition $\tau\gamma\leq-2$ is crucial for this, it automatically follows from the assumption \eqref{eq 1.15} and $(p-2)\alpha<Q-2$ as follows:\\
Relation \eqref{tau gamma rel} gives
$$\frac{1}{\tau}=a\left(\frac{Q-(1+\gamma/as)(ps+(p-2)\alpha)}{pQ}\right)+\frac{1-a}{q}.$$
We have
\begin{multline*}
\frac{aps}{-\tau\gamma}=a\left(\frac{Q-(1+\gamma/as)(ps+(p-2)\alpha)}{(-\gamma/as)Q}\right)+\frac{1-a}{(-\gamma/as)}\frac{p}{q}\\
=a+(1-a)\frac{p}{q}+\frac{p(1+\gamma/as)}{(-\gamma/as)}\left[a\left(\frac{Q-ps-(p-2)\alpha}{pQ}\right)+\frac{1-a}{q}\right]\\
=a+(1-a)\frac{p}{q}+\frac{p(1+\gamma/as)}{(-\gamma/as)}\left[\frac{Q+\tau\gamma\left(1+(p-2)\frac{\alpha}{ps}\right)}{\tau Q}\right].
\end{multline*}
Assumption \eqref{eq 1.15} implies $\frac{aps}{-\tau\gamma}\geq a$, i.e., $-\tau\gamma\leq ps$. Therefore we have $Q-2+\tau\gamma(p-2)\frac{\alpha}{ps}>Q-2-(p-2)\alpha>0$. This together with \eqref{eq 1.15} implies $\tau\gamma<-2$.}
\end{remark}

\begin{proof}[\textbf{Proof of Theorem} \ref{theorem limiting case}]
Let $A_k$, $F_k$ and $G_k$ be as defined earlier. Choose $n_0\in\ZZ$ such that $2^{n_0}<R\leq 2^{n_0+1}$. From equation \eqref{4.2} we get
\begin{equation}\label{eq 4.1}
\int_{A_k}d(\xi)^{\tau\gamma}|z|^{\tau\gamma(p-2)\frac{\alpha}{p s}}|f|^{\tau}d\xi\leq C[f]_{s,p,\alpha,F_k}^{a\tau}\lVert f\rVert_{L^q(F_k)}^{(1-a)\tau}+C|(f)_{F_k}|^{\tau}.
\end{equation}
Observe that, $\ln(4R/d(\xi))\geq (n_0-k+1)\ln2,\,$ in $A_{k}$. Summing with respect to $k$ and using \eqref{3.24} and \eqref{overlap} we obtain
\begin{equation}\label{eq 4.3}
\int_{d(\xi)>2^m}\frac{d(\xi)^{\tau\gamma}|z|^{\tau\gamma(p-2)\frac{\alpha}{ps}}}{\ln^{\tau}(4R/d(\xi))}|f|^{\tau}d\xi\leq C[f]_{s,p,\alpha}^{a\tau}\lVert f\rVert_{L^q(\HH^n)}^{(1-a)\tau}+C\sum_{k=m}^{n_0}\frac{|(f)_{F_k}|^{\tau}}{(n_0-k+1)^\tau}.
\end{equation}
Here we use $\frac{1}{(n_0-k+1)^\tau}\leq1$ in the first term of right hand side. We now apply Lemma \ref{lemma 4.1} in \eqref{4.4} with $c=\left(\frac{n_0-k+1}{n_0-k+\frac{1}{2}}\right)^{\tau-1}$ and obtain
$$|(f)_{F_{k}}|^\tau\leq \frac{(n_0-k+1)^{\tau-1}}{(n_0-k+\frac{1}{2})^{\tau-1}}|(f)_{F_{k+1}}|^{\tau}+2^{\tau-1}(n_0-k+1)^{\tau-1}[f]_{s,p,\alpha,G_k}^{a\tau}\lVert f\rVert_{L^q(G_k)}^{(1-a)\tau},$$
which implies
$$\frac{|(f)_{F_{k}}|^\tau}{(n_0-k+1)^{\tau-1}}-\frac{|(f)_{F_{k+1}}|^{\tau}}{(n_0-k+\frac{1}{2})^{\tau-1}}\leq C[f]_{s,p,\alpha,G_k}^{a\tau}\lVert f\rVert_{L^q(G_k)}^{(1-a)\tau}.$$
Summation with respect to $k$ gives
\begin{equation}
\sum_{k=m}^{n_0}\frac{|(f)_{F_{k}}|^\tau}{(n_0-k+1)^{\tau-1}}-\sum_{k=m}^{n_0}\frac{|(f)_{F_{k+1}}|^{\tau}}{(n_0-k+\frac{1}{2})^{\tau-1}}\leq C[f]_{s,p,\alpha}^{a\tau}\lVert f\rVert_{L^q(\HH^n)}^{(1-a)\tau}.
\end{equation}
On the R.H.S we use \eqref{3.24} and \eqref{overlap}. Using a change of variable in the second sum of L.H.S, we obtain
\begin{multline}
\frac{|(f)_{F_{m}}|^\tau}{(n_0-m+1)^{\tau-1}}+\sum_{k=m+1}^{n_0}\left[\frac{1}{(n_0-k+1)^{\tau-1}}-\frac{1}{(n_0-k+\frac{3}{2})^{\tau-1}}\right]|(f)_{F_{k}}|^\tau\\
\leq C[f]_{s,p,\alpha}^{a\tau}\lVert f\rVert_{L^q(\HH^n)}^{(1-a)\tau}.
\end{multline}
One can observe that $\frac{1}{(n_0-k+1)^{\tau-1}}-\frac{1}{(n_0-k+\frac{3}{2})^{\tau-1}}\sim\frac{1}{(n_0-k+1)^{\tau}}$. Hence
\begin{equation}\label{eq 4.8}
\sum_{k=m}^{n_0}\frac{|(f)_{F_{k}}|^\tau}{(n_0-k+1)^{\tau}}\leq C[f]_{s,p,\alpha}^{a\tau}\lVert f\rVert_{L^q(\HH^n)}^{(1-a)\tau}.
\end{equation}
Inequality \eqref{eq limiting case} follows from \eqref{eq 4.3} and \eqref{eq 4.8} and by passing to the limit $m\rightarrow-\infty$.\\
One can prove \eqref{eq limiting case b} by a similar technique used in the proof of  \eqref{eq frac CKN in H^n ii} and \eqref{eq limiting case}.
\end{proof}
\textbf{Acknowledgements:} 
This work is  part of doctoral thesis of the second author. He is grateful for the support provided by IIT Kanpur, India and MHRD, Government of India (GATE fellowship). 
The third author acknowledges the support provided by SERB through the Core Research Grant (CRG/2022/007867).

\end{document}